\documentclass[12pt]{article}%{amsart}
\usepackage{amstext, amsmath, amsthm, amssymb}
\usepackage{amsmath}
\usepackage{geometry} % see geometry.pdf on how to lay out the page. There's lots.
\usepackage{color}

\numberwithin{equation}{section}
\usepackage[nottoc]{tocbibind}
\usepackage{hyperref}
\usepackage{makeidx}
\makeindex

\newtheorem{thm}{Theorem}[section]
\newtheorem{cor}[thm]{Corollary}
\newtheorem{lem}[thm]{Lemma}
\newtheorem{prop}[thm]{Proposition}
\newtheorem{defin}[thm]{Definition}
\newtheorem{rem}[thm]{Remark}
\newtheorem{example}[thm]{Example}

\newcommand{\R}{\mathbb{R}}

\newcommand{\Z}{\mathbb{Z}}
\newcommand{\E}{\mathbb{E}}
\newcommand{\Ind}{\textbf{1}}
\newcommand{\bi}{{\bf i}}
\newcommand{\bj}{{\bf j}}

\newcommand{\M}{\mathbf{M}}
\newcommand{\G}{\mathbf{G}}
\newcommand{\X}{\mathbb{X}}
\newcommand{\bX}{\mathbf{X}}

\def\liml{\lim\limits}
\def\supl{\sup\limits}
\def\suml{\sum\limits}
\def\intl{\int\limits}
\def\cupl{\mathop{\cup}\limits}

\title{Ergodicity for infinite particle systems with locally conserved quantities.
\thanks{
{Supported by EPSRC %GR/R90994/01 \&
EP/D05379X/1}} %$\phantom{AAAAAAAAAAAAAAAAAAAAAAAAAAAAAAA}$
 }

\author{J. INGLIS$^\S$, M. NEKLYUDOV$^\ddag$ and
B. ZEGARLI\'NSKI$^\sharp${\footnote{ On leave from Imperial College London}}
\\ %}
$^\S$Department of Mathematics, Imperial College London, UK \\ %} %\\
$^\ddag$Department of Mathematics, University of York, Heslington,
UK %E-mail: mn505@york.ac.uk
\\ %}
$^\dag$CNRS, Toulouse, France
}
\date{}

\begin{document}
\maketitle

\begin{abstract}
\noindent We analyse certain degenerate infinite dimensional sub-elliptic generators,
and obtain estimates on the long-time behaviour of the corresponding Markov semigroups that describe a certain model of heat conduction. In particular, we establish
ergodicity of the system for a family of invariant measures, and show that the
optimal rate of convergence to equilibrium is polynomial. Consequently,
there is no spectral gap, but a Liggett-Nash type inequality is shown to hold.
\\

\noindent{Keywords: H\"ormander type generators, locally conserved quantities,  Liggett-Nash inequality, ergodicity.}
\end{abstract}

%\keywords

\tableofcontents

%%%%%%%%%%%%%%%%%%%%%%%%%%%%%%%% Section.1 %%%%%%%%%%%%%%%%%%%%%%%%%%%%%%%%%%%%%%%%%%%%%%%%%%%%%%%%%%%%%%%%%%%%%%%%%%
\section{Introduction}\label{sec:intro} \label{Section.1}
\noindent In this paper we study a class of Markov semigroups  $(P_t)_{t\geq 0}$ whose  generators are defined in H\"ormander form by an infinite family of non-commuting fields as follows

\begin{equation}\label{I.1}
\mathcal{L}\equiv \sum_{\gamma} \bX_\gamma^2.
\end{equation}
In particular we will be interested in the situation when we have ``locally preserved quantities'', that is when any operator
\[
\mathcal{L}_\Lambda\equiv  \sum_{\gamma\in\Lambda} \bX_\gamma^2,
\]
defined with a finite set of indices $\Lambda$, has a non-trivial set of harmonic functions, while for the full generator $\mathcal{L}$, this is not the case.  One should therefore expect that the corresponding semigroup is ergodic. We will assume that the fields $\bX_\gamma$ are homogeneous of the same degree, in the sense that there is a natural dilation generator $D$ such that
\[
[D,\bX_\gamma]=\lambda \bX_\gamma,
\]
with $\lambda\in\mathbb{R}$ independent of $\gamma$. However, unlike in H\"ormander theory, we admit a situation when a commutator of the degenerate fields of any order does not remove degeneration.
To model such a situation we consider an infinite product space %$\Omega\equiv (\R)^{\Z^N}$ and generators
and fields of the following form

\[
%L = \frac{1}{4}\sum_{{\bf i}\in\mathbb{Z}^N}\sum_{{\bf j}\in\mathbb{Z}^N:|{\bf i}-{\bf j}|_1=1}(
\bX_{\bf ij}\equiv\partial_{\bf i}V(x)\partial_{\bf j} - \partial_{\bf j}V(x)\partial_{\bf i},
%)^2.
\]
with $\partial_{\bf i}$ denoting the partial derivative with respect to the coordinate with index ${\bf i}$, and $\partial_{\bf i}V(x)$ indicating some (polynomial) coefficients.

Generators of a similar type appear in the study of dissipative dynamics in which certain quantities are preserved --- see for example \cite{B2007, B-O} and \cite{F-N-O}, where systems of harmonic oscillators perturbed by conservative noise are considered.  A special case of the system we investigate (see Example \ref{ex:GeneratorMainExample} below) can be thought of as a limiting case of the models considered in these works, when the conservative noise dominates the deterministic interaction between oscillators. It is interesting that our results (see Corollary \ref{cor:ConvergenceRate}) show ergodicity of the system even in such situations!
A further example of a physical model very closely related to our set-up is the heat conduction model discussed in \cite{B2008} and \cite{G-K-R}.
For more information in this direction, in particular in connection with an effort to explain the so-called Fourier law of heat conduction, we refer to a nice review \cite{BLRB00}, as well as \cite{BLLO} and the references therein.

The classical approach to studying the asymptotic behaviour of conservative reversible interacting particle systems employs either functional inequalities and some special norm-bound
of the semigroup (see for instance \cite{BZ99, BZ99b} and \cite{PKZ}),
or some kind of approximation of the dynamics by finite dimensional ones, together with sharp estimates of their spectral gaps (\cite{JLQY1999, LY2003}). The approach we take is quite different, in the sense that we do not use any approximation techniques, but rather exploit the structure of the Lie algebra generated by the corresponding vector fields to derive the necessary estimates directly. We would like to note that a straight-forward application of the classical approach in our case is not possible.  This is because any finite dimensional approximation cannot be ergodic, since there is a formal fixed point --- this is discussed in more detail below.  A different possibility would be to consider the restriction of the finite dimensional dynamics to the conservation surface. This approach will be considered elsewhere in a forthcoming manuscript by Z. Brze\'zniak and M. Neklyudov.

%Another line of motivation stems from the works  \cite{BZ99}, \cite{BZ99b} and \cite{PKZ}, where an attempt was made to understand infinite systems from the point of view of functional inequalities. Although formally similar to our present situation, we notice that one can obtain a variety of different long-time behaviours depending on the underlying space.

%The semigroup $(P_t)_{t\geq0}$ that we consider is one of the simplest examples of a semigroup where the family of vector fields corresponding to the generator has a non-trivial Lie bracket which is degenerate (at point $0$).
One other motivation to study the semigroup $(P_t)_{t\geq 0}$ associated to this particular generator comes from the fact that, since $V$ is formally conserved under the action of $P_t$, we can see that there is a family of invariant measures formally given by ``$e^{-\frac{V}{r}}dx$" for all $r>0$.
%As a result, we can formally conclude that the measures $d\mu_r=e^{-\frac{V}{r}}dx,r>0,$ are invariant.
On the one hand, the semigroup $(P_t)_{t\geq0}$ is quite simple, since we can calculate many quantities we are interested in directly. On the other hand, standard methods from interacting particle theory \cite{Lig04, LZ2007} do not help because they require some type of strong non-degeneracy condition such as  H\"{o}rmander's condition, which is not satisfied in our case. Another difficulty stems from the intrinsic difference between the infinite dimensional case we consider, and the finite dimensional case i.e. the case when $V$ depends on only a finite number of variables, and instead of the lattice we use its truncation with a periodic boundary condition. Indeed, as already mentioned, in the finite dimensional case we can notice that $V$ is a non-trivial fixed point for $P_{t}$, and therefore  the semigroup is strictly not ergodic.
%Another natural observation would be that since $V$ is a nontrivial fixed point  for $P_{\cdot}$ than semigroup $P_{\cdot}$ is not ergodic.
This reasoning turns out to be incorrect in the infinite dimensional case. The situation here is more subtle because the expression $V$ is only formal (and would be equal to infinity on the support set of the invariant measure).

We give a detailed study of the case %show that in the case %of $V(x)=\sum_{i\in\mathbb{Z}^N}x_i^2$
when the coefficients of the fields are linear, providing analysis of the corresponding spectral theory and showing that the system is ergodic with polynomial rate of convergence to equilibrium.

The organisation of the paper is as follows. In Section 2 we introduce the basic notation and state an infinite system of stochastic differential equations of interest to us.
In Section 3 we show the existence of a mild solution and continue in Sections 4 and 5 with some discussion of general properties of the corresponding semigroup,
such as the existence of a family of invariant measures, strong continuity, positivity and contractivity properties in $L^p$-spaces. Because of the special non-commutative features of the fields and the form of the generator, these matters are slightly more cumbersome than otherwise.
Section 6 provides a certain characterisation of invariant (Sobolev-type) subspaces, while Section 7 is devoted to the demonstration of ergodicity with optimal rate of convergence to equilibrium. In Section 8 we use previously obtained information to derive Liggett-Nash-type inequalities.  In Section 9 we consider a generalised dynamics
of a similar type, allowing now the inclusion of a first order term $-\beta D$ with some parameter $\beta\in[0,\infty)$ in the generator. We show that in such families one observes a change in the behaviour of the decay to equilibrium from exponential to algebraic (when the additional control parameter $\beta$ goes towards zero).
Finally in the last section we provide a further application of our ergodicity results.

%%%%%%%%%%%%%%%%%%%%%%%%%%%%%%%% Section.2 %%%%%%%%%%%%%%%%%%%%%%%%%%%%%%%%
\section{The system}
\label{The system} \label{Section.2}

Throughout this paper we will work in the following setting.

~

\paragraph{\textit{The Lattice:}}
Let $\Z^N$ be the $N$-dimensional square lattice for some fixed $N\in\mathbb{N}$.  We equip $\Z^N$ with the $l_1$ lattice metric $dist(\cdot, \cdot)$
defined by
\[
dist({\bf i}, {\bf j}) := |{\bf i} -{\bf j}|_1 \equiv \sum_{l=1}^N|i_l - j_l|
\]
for ${\bf i}=(i_1, \dots, i_N), {\bf j} = (j_1, \dots, j_N) \in \Z^N$.  For ${\bf i},{\bf j}\in\Z^N$ we will write $\bi \sim \bj$ whenever
$dist({\bf i}, {\bf j}) =1$.  When $\bi \sim \bj$ we say that $\bi$ and $\bj$ are neighbours in the lattice.

\paragraph{\textit{The Configuration Space:}} Let $\Omega\equiv (\R)^{\Z^N}$.
Define the Hilbert spaces
\[
E_\alpha = \left\{ x \in \Omega : |x|^2_{E_\alpha} := \sum_{{\bf i}\in\Z^N}x_{\bf i}^2e^{-\alpha |{\bf i}|_1} <\infty\right\}
\]
for $\alpha>0$, and
\[
H = \left\{ (h^{(1)}, \dots, h^{(N)}) \in \left(\Omega\right)^N:  |(h^{(1)}, \dots, h^{(N)})|^2_H := \sum_{{\bf i}\in\Z^N}
\sum_{k=1}^N\left(h_{\bf i}^{(k)}\right)^2 < \infty\right\},
\]
with inner products given by
\[
\langle x, y\rangle_{E_\alpha} := \sum_{{\bf i} \in\Z^N}x_{\bf i}y_{\bf i}e^{-\alpha|{\bf i}|_1}
\]
for $x, y\in E_\alpha$ and
\[
\langle (g^{(1)}, \dots, g^{(N)}), (h^{(1)}, \dots, h^{(N)})\rangle_H : = \sum_{{\bf i}\in\Z^N} \sum_{k=1}^Ng_{\bf i}^{(k)}h_{\bf i}^{(k)}
\]
for $ (g^{(1)}, \dots, g^{(N)}), (h^{(1)}, \dots, h^{(N)})\in H$ respectively.

\paragraph{\textit{The Gibbs Measure:}}

Let $\mu_{\bf G}$ be a Gaussian probability measure on $(E_\alpha, \mathcal{B}(E_\alpha))$ with mean zero and covariance ${\bf G}$.  We assume that
the inverse ${\bf G}^{-1}$ of the covariance is of finite range i.e.
\[
{\bf M}_{{\bf i},{\bf j}} := {\bf G}^{-1}_{{\bf i},{\bf j}} = 0\quad \mathrm{if}\ dist({\bf i},{\bf j})>R,
\]
and that $|{\bf M}_{{\bf i},{\bf j}}|\leq M$ for all ${\bf i},{\bf j}\in\Z^N$.

\paragraph{\textit{The System:}}

Let
\[
W = \left\{\left(W^{(1)}, \dots, W^{(N)}\right)\right\}
\]
be a cylindrical Wiener process in $H$ (see for instance \cite{D-Z}).

We introduce the following notation: for ${\bf i}=(i_1, \dots, i_N)\in\Z^N$ define for $k\in\{1, \dots, N\}$
\[
{\bf i}^{\pm}(k) := (i_1, \dots, i_{k-1}, i_k \pm 1, i_{k+1}, \dots, i_N).
\]
We also define, for $x\in E_\alpha$, ${\bf i}\in\Z^N$,
\[
V_{\bf i}(x) := \sum_{{\bf j}\in\Z^N}x_{\bf i}{\bf M}_{{\bf i},{\bf j}}x_{\bf j},
\]
which is finite since ${\bf M}_{{\bf i},{\bf j}} =0$ if $dist({\bf i}, {\bf j})>R$, and for all finite subsets $\Lambda \subset \Z^N$ set
\[
V_{\Lambda}(x) := \sum_{{\bf i}\in\Lambda} V_{\bf i}(x).
\]
Using the formal expression
\[
V(x) := \frac12\sum_{{\bf i}\in\mathbb{Z}^N} V_{\bf i}(x),
\]
it will be convenient to simplify the notation for  $\partial_{\bf i} V_{\bf i}$ as follows
\begin{align*}
\partial_{\bf i}V(x)  = \frac12\partial_{\bf i}\left(\sum_{{\bf j},{\bf l}\in\mathbb{Z}^N}x_{\bf j} {\bf M}_{{\bf j},{\bf l}}x_{\bf l}\right)
\equiv  \sum_{{\bf j}\in\Z^N}{\bf M}_{{\bf i},{\bf j}}x_{\bf j}= \partial_{\bf i} V_{\bf i}.
\end{align*}

We consider the following system of Stratonovich SDEs:
\begin{align}
\label{SDE system strat}
dY_{\bf i}(t) & = \sum_{k=1}^N\left(\partial_{{\bf i}^-(k)}V(Y(t))\circ dW_{{\bf i}^-(k)}^{(k)}(t) - \partial_{{\bf i}^+(k)}V(Y(t))\circ dW_{{\bf
i}}^{(k)}(t)\right),
\end{align}
for ${\bf i}\in\Z^N$ and $t\geq0$.

%%%%%%%%%%%%%%%%%%%%%%%%%%%%%%%% Section.3 %%%%%%%%%%%%%%%%%%%%%%%%%%%%%%%%
\section{Existence of a  mild solution} \label{Section.3}
\label{solution}
In this section we show that the system \eqref{SDE system strat} has a mild solution $Y(t)$ taking values in the Hilbert space $E_\alpha$.
%, before describing some of its properties.

For the existence of a mild solution, the first step is to write \eqref{SDE system strat} in It\^o form.  To this end we have
\begin{align}
\label{Ito1}
dY_{\bf i}(t) & = \sum_{k=1}^N\left(\partial_{{\bf i}^-(k)}V(Y(t)) dW_{{\bf i}^-(k)}^{(k)}(t) - \partial_{{\bf i}^+(k)}V(Y(t)) dW_{\bf
i}^{(k)}(t)\right) \nonumber \\
& \quad + \frac{1}{2}\sum_{k=1}^N\left( d\left[\partial_{{\bf i}^{-}(k)}V(Y(\cdot)), W^{(k)}_{{\bf i}^-(k)}(\cdot)\right]_t - d\left[\partial_{{\bf
i}^{+}(k)}V(Y(\cdot)), W^{(k)}_{\bf i}(\cdot)\right]_t \right)
\end{align}
for all $i\in\Z^N$ and $t\geq 0$, where $[\cdot,\cdot]_t$ is a quadratic covariation (see for example p.~61 of \cite{[Ikeda-1981]}).
Hence, by It\^o's formula,
\begin{align}
&\left[\partial_{{\bf i}^{-}(k)}V(Y(\cdot)), W^{(k)}_{{\bf i}^-(k)}(\cdot)\right]_t = \left[\sum_{{\bf j}\in\Z^N}\int_0^\cdot\partial_{\bf j}
\partial_{{\bf i}^{-}(k)}V(Y(s))dY_{\bf j}(s), \int_0^\cdot dW^{(k)}_{{\bf i}^{-}(k)}(s)\right]_t \nonumber \\
&\quad = \sum_{{\bf j}\in\Z^N}\left[\int_0^\cdot \partial_{\bf j}\partial_{{\bf i}^{-}(k)}V(Y(s))\partial_{{\bf j}^-(k)}V(Y(s))dW_{{\bf
j}^-(k)}^{k}(s), \int_0^\cdot dW_{{\bf i}^-(k)}^k(s)\right]_t \nonumber\\
&\qquad -  \sum_{{\bf j}\in\Z^N}\left[\int_0^\cdot \partial_{\bf j}\partial_{{\bf i}^{-}(k)}V(Y(s))\partial_{{\bf j}^+(k)}V(Y(s))dW_{\bf j}^{k}(s),
\int_0^\cdot dW_{{\bf i}^-(k)}^k(s)\right]_t \nonumber\\
& \quad = \int_0^t \partial^2_{{\bf i},{\bf i}^-(k)}V(Y(s))\partial_{{\bf i}^-(k)}V(Y(s))ds%\nonumber\\
%& \qquad
- \int_0^t\partial^2_{{\bf i}^{-}(k)}V(Y(s))\partial_{\bf i}V(Y(s))ds.\nonumber
\end{align}
By a similar calculation, and using this in \eqref{Ito1}, we see that
\begin{align}
\label{Ito2}
dY_{\bf i}(t) & = \sum_{k=1}^N\left(\partial_{{\bf i}^-(k)}V(Y(t)) dW_{{\bf i}^-(k)}^{(k)}(t) - \partial_{{\bf i}^+(k)}V(Y(t)) dW_{\bf
i}^{(k)}(t)\right)\nonumber \\
& \quad - \frac{1}{2}\sum_{k=1}^N\left\{\left(\partial^2_{{\bf i}^-(k)}V(Y(t))+ \partial^2_{{\bf i}^+(k)}V(Y(t))\right)\partial_{\bf i}V(Y(t))\right.
\nonumber \\
& \quad \qquad - \left.\partial^2_{{\bf i}, {\bf i}^-(k)}V(Y(t))\partial_{{\bf i}^-(k)}V(Y(t))  - \partial^2_{{\bf i},{\bf
i}^+(k)}V(Y(t))\partial_{{\bf i}^+(k)}V(Y(t))\right\}dt
\end{align}
for all $i\in\Z^N$.

Recall now that $\partial_{\bf j}V(x) = \sum_{{\bf l}\in\Z^N}{\bf M}_{{\bf j}, {\bf l}}x_{\bf l}$ for all $j\in\Z^N$, so that $\partial^2_{{\bf i},{\bf j}}V(x) =  {\bf
M}_{{\bf i},{\bf j}}$, $\forall i,j\in\Z^N$.  Thus the system \eqref{Ito2} can be written as
\begin{align}
\label{Ito3}
dY_{\bf i}(t) & = \sum_{k=1}^N\left(\partial_{{\bf i}^-(k)}V(Y(t)) dW_{{\bf i}^-(k)}^{(k)}(t) - \partial_{{\bf i}^+(k)}V(Y(t)) dW_{\bf
i}^{(k)}(t)\right)\nonumber \\
& \quad - \frac{1}{2}\sum_{k=1}^N\Big\{\left({\bf M}_{{\bf i}^-(k), {\bf i}^-(k)} + {\bf M}_{{\bf i}^+(k), {\bf i}^+(k)}\right)\partial_{\bf i}V(Y(t))
\nonumber \\
& \quad \qquad - {\bf M}_{{\bf i}, {\bf i}^-(k)}\partial_{{\bf i}^-(k)}V(Y(t))  - {\bf M}_{{\bf i}, {\bf i}^+(k)}\partial_{{\bf
i}^+(k)}V(Y(t))\Big\}dt
\end{align}
for all $i\in\Z^N$ and $t\geq0$.
We now claim that we can write this system in operator form:
\begin{equation}
\label{Ito op form}
dY(t) = AY(t)dt + B(Y(t))dW(t),
\end{equation}
where $A$ is a bounded linear mapping from $E_\alpha$ to $E_\alpha$ given by
\begin{align}
\label{A}
(Ax)_{\bf i} &:= \sum_{k=1}^{N}a_{\bf i}^{(k)}(x), \quad {\bf i} \in \Z^N,
\end{align}
with
\begin{align}
\label{alpha}
a_{\bf i}^{(k)}(x) & = - \frac{1}{2}\left\{\left({\bf M}_{{\bf i}^-(k), {\bf i}^-(k)} + {\bf M}_{{\bf i}^+(k), {\bf i}^+(k)}\right)\sum_{{\bf
l}\in\Z^N}{\bf M}_{{\bf l}, {\bf i}}x_{\bf l}\right.  \nonumber \\
& \quad \qquad - \left.{\bf M}_{{\bf i}, {\bf i}^-(k)}\sum_{{\bf l}\in\Z^N}{\bf M}_{{\bf l}, {\bf i}^-(k)}x_{\bf l}  - {\bf M}_{{\bf i}, {\bf
i}^+(k)}\sum_{{\bf l}\in\Z^N} {\bf M}_{{\bf l}, {\bf i}^+(k)}x_{\bf l}\right\},
\end{align}
and $B: E_\alpha \to L_{HS}(H, E_\alpha)$ (here $L_{HS}(H, E_\alpha)$ denotes the space of Hilbert-Schmidt operators from $H$ to $E_\alpha$) is a bounded linear operator given by
\begin{align}
\label{B}
\left(B(x)(h^{(1)}, \dots, h^{(N)})\right)_{\bf i} := \sum_{k=1}^N \left(\partial_{{\bf i}^-(k)}V(x)h^{(k)}_{{\bf i}^-(k)} - \partial_{{\bf
i}^+(k)}V(x)h^{(k)}_{\bf i}\right)
\end{align}
for $x\in E_\alpha, (h^{(1)}, \dots, h^{(N)})\in H$ and ${\bf i}\in\Z^N$.

Indeed, the fact that $A:E_\alpha \to E_\alpha$ is a bounded linear operator follows from the fact that the constants ${\bf M}_{{\bf i},{\bf j}}$ are
assumed to be uniformly bounded.  To show that $B\in L(E_\alpha, L_{HS}(H, E_\alpha))$, first define, for ${\bf i}\in\Z^N$, $e({\bf i})\in\Omega$ by
\[
(e({\bf i}))_{\bf j} :=
\begin{cases}
1  & \mathrm{if}\ {\bf j}={\bf i}, \\
0 & \mathrm{otherwise},
\end{cases}
\]
and for ${\bf i}\in\Z^N, k\in\{1, \dots, N\}$, let $f_{\bf i}^k$ be the element in $H$ given by
\[
f_{\bf i}^k := (0,\dots, e({\bf i}), \dots, 0),
\]
where the $e({\bf i})$ occurs in the $k$-th coordinate.  Then
\[
\left\{f_{\bf i}^k: {\bf i}\in\Z^N, k\in\{1, \dots, N\}\right\}
\]
is an orthonormal basis for $H$.  Let $x\in E_\alpha$.  Then
\[
\|B(x)\|^2_{HS} = \sum_{{\bf i}\in\Z^N}\sum_{k=1}^N\left|B(x)(f^k_{\bf i})\right|^2_{E_\alpha}.
\]
Now by definition
\[
\left(B(x)(f^k_{\bf i})\right)_{\bf j} = \partial_{{\bf j}^-(k)}V(x)\left(e({\bf i})\right)_{{\bf j}^-(k)} - \partial_{{\bf j}^+(k)}V(x)(e({\bf
i}))_{\bf j}
\]
so that
\begin{align}
\left|B(x)(f_{\bf i}^k)\right|^2_{E_\alpha} &= \sum_{{\bf j}\in\Z^N}\left(\partial_{{\bf j}^-(k)}V(x)\left(e({\bf i})\right)_{{\bf j}^-(k)} -
\partial_{{\bf j}^+(k)}V(x)(e({\bf i}))_{\bf j}\right)^2e^{-\alpha|{\bf j}|_1} \nonumber \\
&= \left(\partial_{\bf i}V(x)\right)^2e^{-\alpha|{\bf i}^+(k)|_1} + \left(\partial_{{\bf i}^+(k)}V(x)\right)^2e^{-\alpha|{\bf i}|_1}\nonumber \\
%&= 4\left(\sum_{{\bf l}:|{\bf l}-{\bf i}|_1\leq R}{\bf M}_{{\bf i},{\bf l}}x_{\bf l}\right)^2e^{-\alpha|{\bf i}^+(k)|_1} + 4\left(\sum_{{\bf l}:|{\bf
%l}-{\bf i}^+(k)|_1 \leq R}{\bf M}_{{\bf i}^+(k), {\bf l}}x_{\bf l}\right)^2e^{-\alpha|{\bf i}|_1}\nonumber \\
%&\leq C\left[\left(\sum_{{\bf l}:|{\bf l}-{\bf i}|_1\leq R}x_{\bf l}^2\right)e^{-\alpha|{\bf i}^+(k)|_1} + \left(\sum_{{\bf l}:|{\bf l}-{\bf
%i}^+(k)|_1 \leq R}x_{\bf l}^2\right)e^{-\alpha|{\bf i}|_1}\right]\\
%&\leq Ce^\alpha\left[\left(\sum_{{\bf l}:|{\bf l}-{\bf i}|_1\leq R}x_{\bf l}^2\right)e^{-\alpha|{\bf i}|_1} + \left(\sum_{{\bf l}:|{\bf l}-{\bf
%i}^+(k)|_1 \leq R}x_{\bf l}^2\right)e^{-\alpha|{\bf i}^+(k)|_1}\right]\\
&\leq Ce^{(R+1)\alpha}\left[\left(\sum_{{\bf l}:|{\bf l}-{\bf i}|_1\leq R}x_{\bf l}^2e^{-\alpha|{\bf l}|_1}\right) + \left(\sum_{{\bf l}:|{\bf
l}-{\bf
i}^+(k)|_1 \leq R}x_{\bf l}^2e^{-\alpha|{\bf l}|_1}\right)\right]\nonumber %\\
\end{align}
where $C=((2R)^N+1)M^2$. Thus
\begin{align*}
\|B(x)\|_{HS} &= \sum_{{\bf i}\in\Z^N}\sum_{k=1}^N\left|B(x)(f^k_{\bf i})\right|^2_{E_\alpha} \\
&\leq Ce^{(R+1)\alpha}\sum_{k=1}^N\sum_{{\bf i}\in\Z^N}\left[\left(\sum_{{\bf l}:|{\bf l}-{\bf i}|_1\leq R}x_{\bf l}^2e^{-\alpha|{\bf l}|_1}\right) +
\left(\sum_{{\bf l}:|{\bf l}-{\bf i}^+(k)|_1 \leq R}x_{\bf l}^2e^{-\alpha|{\bf l}|_1}\right)\right] \\
%&= 2((2R)^N+1)Ce^{(R+1)\alpha}\sum_{k=1}^N\left(\sum_{{\bf i}\in\Z^N}x_{\bf i}^2e^{-\alpha|{\bf i}|_1}\right)\\
&= 2N((2R)^N+1)Ce^{(R+1)\alpha}|x|^2_{E_\alpha},
\end{align*}
which proves the claim that $B\in L(E_\alpha, L_{HS}(H, E_\alpha))$.

We thus have the following existence theorem for our system.
\begin{prop}
\label{existence}
Consider the stochastic evolution equation
\begin{equation}
\label{system}
dY(t) = AY(t)dt + B(Y(t))dW(t), \qquad Y_0 = x\in E_\alpha,\ t\geq0,
\end{equation}
where $A$ and $B$ are given by \eqref{A} and \eqref{B} respectively, and $(W(t))_{t\geq0}$ is a cylindrical Wiener process in $H$.  This equation has
a mild solution $Y$ taking values in the Hilbert space $E_\alpha$, unique up to equivalence among the processes satisfying
\[
\mathbb{P}\left(\int_0^T|Y(s)|_{E_\alpha}^2ds < \infty\right) = 1.
\]
Moreover, it has a continuous modification.
\end{prop}

\begin{proof}
We have shown above that $A:E_\alpha\to E_\alpha$ is a bounded linear operator, so that it is the infinitesimal generator of a $C_0$-semigroup on
$E_\alpha$ ($A$ can be thought of as a bounded linear perturbation of $0$, which is trivially the generator of a $C_0$-semigroup).  We have also shown
that $B\in L(E_\alpha, L_{HS}(H,E_\alpha))$.  Hence the result follows immediately from Theorem 7.4 of \cite{D-Z}.
\end{proof}

\begin{lem}
\label{generator}
The mild solution $Y$ to \eqref{system} solves the martingale problem for the operator
\[
\mathcal{L} = \frac{1}{4}\sum_{{\bf i}\in\Z^N}\sum_{{\bf j}\in\Z^N:{\bf i}\sim{\bf j}}(\partial_{\bf i}V(x)\partial_{\bf j} - \partial_{\bf
j}V(x)\partial_{\bf i})^2.
\]
\end{lem}

\begin{proof}

By It\^o's formula, we have for any suitable function $f$ that
\begin{align*}
f(Y(t))& = f(Y(0)) + \sum_{{\bf i}\in\Z^N}\int_0^t\partial_{\bf i}f(Y(s))dY_{\bf i}(s) \\
&\qquad + \frac{1}{2}\sum_{{\bf i},{\bf j}\in\Z^N}\int_0^t\partial^2_{{\bf i}, {\bf j}}f(Y(s))d\left[Y_{\bf i}, Y_{\bf j}\right]_s.
\end{align*}
We can then calculate from \eqref{Ito3} that
\[
d[Y_{\bf i}, Y_{\bf j}]_t :=\\
\begin{cases}
-\partial_{\bf i}V(Y(t))\partial_{{\bf i}^-(k)}V(Y(t))dt   & \mathrm{if}\ {\bf j}={\bf i}^-(k),\\
\sum_{k=1}^N\left\{ \left(\partial_{{\bf i}^-(k)}V(Y(t))\right)^2 + \left(\partial_{{\bf i}^+(k)}V(Y(t))\right)^2\right\}dt & \mathrm{if}\ {\bf j}={\bf i},\\
-\partial_{\bf i}V(Y(t))\partial_{{\bf i}^+(k)}V(Y(t))dt & \mathrm{if}\  {\bf j}={\bf i}^-(k),\\
\end{cases}
\]
so that
\begin{align*}
&\sum_{{\bf i},{\bf j}\in\Z^N}\int_0^t\partial^2_{{\bf i}, {\bf j}}f(Y(s))d\left[Y_{\bf i}, Y_{\bf j}\right]_s \\
&  = \sum_{{\bf i}\in\Z^N}\int_0^t\partial^2_{\bf i}f(Y(s))\sum_{k=1}^N\left\{ \left(\partial_{{\bf i}^-(k)}V(Y(t))\right)^2 + \left(\partial_{{\bf
i}^+(k)}V(Y(t))\right)^2\right\}dt \\
& -  2\sum_{k=1}^N\sum_{{\bf i}\in\Z^N}\int_0^t\partial^2_{{\bf i},{\bf i}^-(k)}f(Y(s))\partial_{\bf i}V(Y(t))\partial_{{\bf i}^-(k)}V(Y(t))dt. \\
\end{align*}
Thus, using \eqref{Ito2}, the generator of the system is given by
\begin{align*}
\mathcal{L} &= \frac{1}{2}\sum_{{\bf i}\in\Z^N}\sum_{k=1}^N\left\{ \left(\partial_{{\bf i}^-(k)}V(x)\right)^2 + \left(\partial_{{\bf
i}^+(k)}V(x)\right)^2\right\}\partial^2_{\bf i} \\
& - \sum_{{\bf i}\in\Z^N}\sum_{k=1}^N \partial_{\bf i}V(x)\partial_{{\bf i}^-(k)}V(x)\partial^2_{{\bf i}, {\bf i}^-(k)}\\
& - \frac{1}{2}\sum_{{\bf i}\in\Z^N}\sum_{k=1}^N\left\{\left(\partial^2_{{\bf i}^-(k)}V(x)+ \partial^2_{{\bf i}^+(k)}V(x)\right)\partial_{\bf
i}V(x)\right.  \nonumber \\
& \quad \qquad - \left.\partial^2_{{\bf i}, {\bf i}^-(k)}V(x)\partial_{{\bf i}^-(k)}V(x)  - \partial^2_{{\bf i},{\bf i}^+(k)}V(x)\partial_{{\bf
i}^+(k)}V(x)\right\}\partial_{\bf i}.
\end{align*}
One can then check by direct calculation that we have
\[
\mathcal{L} = \frac{1}{4}\sum_{{\bf i}\in\Z^N}\sum_{{\bf j}\in\Z^N:{\bf i}\sim{\bf j}}\left(\partial_{\bf i}V(x)\partial_{\bf j} - \partial_{\bf
j}V(x)\partial_{\bf i}\right)^2.
\]
\end{proof}

For $n\in \{0,1,\dots\}$, let $\mathcal{U}C_b^n \equiv\mathcal{U}C_b^n(E_{\alpha}),\alpha>0$ denote the set of all functions which are uniformly continuous and bounded, together with their Fr\'echet derivatives up to order $n$.

\begin{cor}
\label{rep}
The semigroup $(P_t)_{t\geq0}$ acting on $\mathcal{U}C_b(E_{\alpha}),\alpha>0$ corresponding to the system \eqref{system} is Feller and can be represented
by the formula
\[
P_tf(\cdot) = \E f\left(Y(t, \cdot)\right), \quad t\geq0,
\]
where $Y(t, x)$ is a mild solution to the system \eqref{system} with initial condition $x\in E_{\alpha}$.
Furthermore, $(P_t)_{t\geq0}$ satisfies Kolmogorov's backward equation, and solutions of the system are strong Markov processes.
\end{cor}

\begin{proof}
The result follows immediately from Theorems 9.14 and 9.16 of \cite{D-Z}.
\end{proof}

\begin{example}
\label{ex:GeneratorMainExample}
Suppose that, for all $i\in\Z^N$,
\[
{\bf M}_{{\bf i},{\bf i}} = 1, \quad {\bf M}_{{\bf i},{\bf j}} =0\quad \mathrm{if}\quad {\bf i}\neq {\bf j}.
\]
Then $\partial_{\bf i}V(x) = x_{\bf i}$, and the system \eqref{Ito3} becomes
\[
dY_{\bf i}(t)  = -\sum_{k=1}^N Y_{\bf i}(t)dt + \sum_{k=1}^N \left(Y_{{\bf i}^-(k)}(t)dW^{k}_{{\bf i}^-(k)}(t) - Y_{{\bf i}^+(k)}(t)dW^{k}_{\bf
i}(t)\right)
\]
for all $i\in\Z^N$, which has generator
\begin{equation}\label{eqn:GeneratorMainExample}
\mathcal{L} = \frac{1}{4}\sum_{{\bf i}\in\Z^N}\sum_{{\bf j}\in\Z^N:{\bf i}\sim{\bf j}}\left(x_{\bf i}\partial_{\bf j} - x_{\bf j}\partial_{\bf
i}\right)^2.
\end{equation}
Very closely related generators are considered in the physical models for heat conduction described in \cite{B2007, B2008, B-O, F-N-O} and \cite{G-K-R}. A related model is also considered \cite{Carmona}.  However, there are some major differences between the system considered there and the one we investigate. Indeed, in \cite{Carmona} H\"ormander's condition is assumed to be satisfied, and the system is finite dimensional.  Moreover, it is shown that there is a unique invariant measure for such a system, which as we will see, is not the case in our set-up.

%In this case the Gaussian measure $\mu_{r\bf G}$ on $E_\alpha$ is a product Gaussian measure.
\end{example}

\begin{rem}
Let $(r_{\bf i,j},\theta_{\bf i,j})$ be polar coordinates in the plane  $(x_{\bf i},x_{\bf j})$. Then
\[
\frac{\partial}{\partial \theta_{\bf i,j}}=x_{\bf i}\partial_{\bf j} - x_{\bf j}\partial_{\bf
i}.
\]
Therefore in Example \ref{ex:GeneratorMainExample}
\[
\mathcal{L} = \frac{1}{4}\sum_{{\bf i}\in\Z^N}\sum_{{\bf j}\in\Z^N:{\bf i}\sim{\bf j}}\frac{\partial^2}{\partial \theta_{\bf i,j}^2}.
\]
Note that the operator $-\frac{\partial^2}{\partial \theta_{\bf i,j}^2}$ is the Hamiltonian for the rigid rotor on the plane. Thus, the operator $-\mathcal{L}$ is the Hamiltonian of a chain of coupled rigid rotors.
\end{rem}
\section{Invariant measure}\label{Section.4}
\label{Invariant measure}

Suppose now that $(Y(t))_{t\geq0}$ is the unique mild solution to the evolution
equation \eqref{system} in the Hilbert space $E_\alpha$  i.e.
\[
dY(t) = AY(t)dt + B(Y(t))dW(t)
\]
where $A, B$ are given by \eqref{A} and \eqref{B} respectively, and $(W(t))_{t\geq0}$ is a cylindrical Wiener process in $H$.  Let $(P_t)_{t\geq0}$ be the corresponding semigroup, defined as above.

For ${\bf i}, {\bf j}\in\Z^N$, define
\[
\bX_{{\bf i},{\bf j}} = \partial_{\bf i}V(x)\partial_{\bf j} - \partial_{\bf j}V(x)\partial_{\bf i}
\]
so that by Lemma \ref{generator},
\[
\mathcal{L} = \frac{1}{4}\sum_{{\bf i}\in\Z^N}\sum_{{\bf j}\in\Z^N:{\bf i}\sim{\bf j}}\bX_{{\bf i},{\bf j}}^2
\]
is the generator of our system.
We will need the following Lemma:

\begin{lem}
\label{antisym}
\[
\mu_{r\bf G}\left(f\bX_{{\bf i},{\bf j}}g\right) = - \mu_{r\bf G}\left(g\bX_{{\bf i},{\bf j}}f\right)
\]
for all $f,g\in \mathcal{U}C_b^2(E_{\alpha})$, ${\bf i},{\bf j}\in\Z^N$ and $r>0$.
\end{lem}

\begin{proof}
For finite subsets $\Lambda\subset\Z^N$ and $\omega\in\R^{\Z^N}$, denote by $\E_\Lambda^\omega$ the conditional measure of $\mu_{r\bf G}$, given the
coordinates outside $\Lambda$ coincide with those of $\omega$.  Then we have that
\[
\E_\Lambda^\omega(f) = \int_{\R^\Lambda}f(x_\Lambda\cdot\omega_{\Lambda^c})\frac{e^{-\frac{1}{2r}\suml_{{\bf k}\in\Lambda}V_{\bf
k}(x_\Lambda\cdot\omega_{\Lambda^c})}}{Z_\Lambda^\omega}dx_\Lambda
\]
where $x_\Lambda\cdot\omega_{\Lambda^c}$ is the element of $\R^{\Z^N}$ given by
\[
(x_\Lambda\cdot\omega_{\Lambda^c})_{\bf i}=
\begin{cases}
x_{\bf i}  & \mathrm{if}\ {\bf i}\in\Lambda, \\
\omega_{\bf i} & \mathrm{if}\ {\bf i}\in\Lambda^c,
\end{cases}
\]
and $Z_\Lambda^\omega$ is the normalisation constant.
Now fix ${\bf i},{\bf j}\in\Z^N$ and suppose that $\Lambda$ is such that $\{{\bf i},{\bf j}\}\subset\Lambda$.  Then for  $f,g\in
\mathcal{U}C_b^2(E_{\alpha})$
\begin{align*}
\E_\Lambda^\omega\left(f\bX_{{\bf i},{\bf j}}g\right) &= \int_{\R^\Lambda}f(x_\Lambda\cdot\omega_{\Lambda^c})
\bX_{\bf i,j}
g(x_\Lambda\cdot\omega_{\Lambda^c})\frac{e^{-\frac{1}{2r}\suml_{{\bf k}\in\Lambda}V_{\bf k}(x_\Lambda\cdot\omega_{\Lambda^c})}}{Z_\Lambda^\omega}dx_\Lambda\\
&= - \E_\Lambda^\omega\left(g \bX_{{\bf i},{\bf j}}f\right) \\
&\qquad +\E_\Lambda^\omega\left(fg\left[\partial_{\bf i}\partial_{\bf j}V(x) - \partial_{\bf j}\partial_{\bf i}V(x)\right]\right)\\
&\qquad +\frac{1}{r}\E_\Lambda^\omega\left(fg\left[\partial_{\bf i}V(x)\partial_{\bf j}V(x) - \partial_{\bf j}V(x)\partial_{\bf i}V(x)\right]\right)%\\
%&
=  - \E_\Lambda^\omega\left(g\bX_{{\bf i},{\bf j}}f\right)
\end{align*}
by integration by parts.  Thus we have that
\[
\mu_{r\bf G}\left(f\bX_{{\bf i},{\bf j}}g\right) = \mu_{r\bf G}\E_\Lambda^\cdot\left(f\bX_{{\bf i},{\bf j}}g\right) = -\mu_{r\bf G}\E_\Lambda^\cdot\left(g
\bX_{{\bf i},{\bf j}}
f\right)
= -\mu_{r\bf G}\left(g\bX_{{\bf i},{\bf j}}f\right).
\]

\end{proof}

The following result shows that for $r>0$, $\mu_{r\bf G}$  is reversible for the system \eqref{Ito op form}.

\begin{thm}
\label{symmetric measure}
For all $f,g\in \mathcal{U}C_b^2(E_{\alpha})$ and $r>0$, we have
\begin{equation}
\label{reversible}
\mu_{r\bf G}\left(fP_tg\right) = \mu_{r\bf G}\left(gP_tf\right).
\end{equation}
\end{thm}

\begin{proof}
It is enough to show that \eqref{reversible} holds for $f,g\in \mathcal{U}C_b^2(E_{\alpha})$ depending only on a finite number of coordinates.
%{\color{red} (maybe another sentence of explanation would be good here?)Done.}
Indeed, we can find sequences of cylindrical functions
$\{f_n\}_{n=1}^{\infty},\{g_n\}_{n=1}^{\infty}\subset\mathcal{U}C_b^2(E_{\alpha})$ which approximate general $f, g\in\mathcal{U}C_b^2(E_{\alpha})$.
% such that $|f_n-f|_{C(E_{\alpha})}+|g_n-g|_{C(E_{\alpha})}\to 0, n\to\infty$ and use the fact that $\forall t>0$
%$$
%|P_t|_{L(C(E_{\alpha}),C(E_{\alpha}))}\leq 1.
%$$
In view of this, suppose $f(x) =
f\left(\{x_{\bf i}\}_{|{\bf i}|_1\leq n}\right)$ and $g(x) = g\left(\{x_{\bf i}\}_{|{\bf i}|_1\leq n}\right)$ for some $n$.
%Define
%\begin{align*}
%X_{{\bf i},{\bf j}} &:=\partial_{\bf i}V(x)\partial_{\bf j} - \partial_{\bf j}V(x)\partial_{\bf i} %\\
%%&
%= \sum_{{\bf l}\in\Z^N}{\bf M}_{\bf i,l}x_{\bf l}\partial_{\bf j} -   \sum_{{\bf l}\in\Z^N}{\bf M}_{\bf j,l}x_{\bf l}\partial_{\bf i}
%\end{align*}
%and recall that the generator of our system is given by
%\[
%\mathcal{L} = \frac{1}{4}\sum_{{\bf i}\in\Z^N}\sum_{{\bf j}:|{\bf j}-{\bf i}|_1=1}X_{{\bf i},{\bf j}}^2.
%\]
Note that the generator $\mathcal{L}$ can be rewritten as
\begin{align*}
\mathcal{L} %&= \frac{1}{2}\sum_{k=1}^N\sum_{{\bf i}\in\Z^N} \left(
%\partial_{{\bf i}^+(k)}V(x)\partial_{\bf i} - \partial_{{\bf i}^+(k)}V(x)\partial_{\bf i}\right)^2\\
&= \frac{1}{2}\sum_{k=1}^N\sum_{{\bf i}\in\Z^N} \bX_{{\bf i}, {\bf i}^+(k)}^2.
\end{align*}

We decompose this operator further.  Indeed, we can write

\begin{align*}
\mathcal{L} = \frac{1}{2}\sum_{k=1}^N\sum_{{\bf m}\in\{0, \dots, R+1\}^N} \left(\sum_{{\bf i}\in\otimes_{\sigma=1}^N\left((R+2)\Z +
m_\sigma\right)} \bX_{{\bf i},{\bf i}^+(k)}^2\right)
\end{align*}
and define for ${\bf m} = (m_1, \dots, m_N)\in\{0,...,R+1\}^N$, $k\in\{1, \dots, N\}$
\[
\mathcal{L}_{\bf m}^{(k)} :=  \sum_{{\bf i}\in\otimes_{\sigma=1}^N\left((R+2)\Z + m_\sigma\right)}\bX_{{\bf i},{\bf i}^+(k)}^2
\]
so that
\[
\mathcal{L} = \frac{1}{2}\sum_{k=1}^N\sum_{{\bf m}\in\{0, \dots, R+1\}^N}\mathcal{L}_{\bf m}^{(k)}.
\]
Note that by construction, for fixed $k\in\{1, \dots, N\}$ and ${\bf m}\in\{0,...,R+1\}^N$, we have for any ${\bf i},{\bf j} \in
\otimes_{\sigma=1}^N\left((R+2)\Z + m_\sigma\right)$ that
\[
\left[\bX_{{\bf i}, {\bf i}^+(k)}, \bX_{{\bf j}, {\bf j}^+(k)}\right] =0.
\]
For ${\bf i}={\bf j}$ this is clear.  If ${\bf i}\neq {\bf j}$, we have
\[
\left[\bX_{{\bf i}, {\bf i}^+(k)}, \bX_{{\bf j}, {\bf j}^+(k)}\right] = \left[ \partial_{\bf i}V(x)\partial_{{\bf i}^+(k)} - \partial_{{\bf
i}^+(k)}V(x)\partial_{\bf i}, \partial_{\bf j}V(x)\partial_{{\bf j}^+(k)} - \partial_{{\bf j}^+(k)}V(x)\partial_{\bf j}\right]
\]
and
\[
\partial_{{\bf i}^+(k)}\partial_{\bf j}V(x) =0.
\]
Indeed, $\partial_{\bf j}V(x)$ depends only on coordinates ${\bf l}$ such that $|{\bf j}-{\bf l}|_1\leq R$, and for all such ${\bf l}$
\begin{align*}
|{\bf i}^+(k) - {\bf l}|_1 &\geq |{\bf i}^+(k)-{\bf j}|_1 - |{\bf j} - {\bf l}|_1\\
&\geq R+1 - R\\
&=1
\end{align*}
so that $\partial_{\bf j}V(x)$ does not depend on coordinate ${\bf i}^+(k)$ for any $k$.  Similarly
\[
\partial_{{\bf i}^+(k)}\partial_{{\bf j}^+(k)}V(x) = \partial_{\bf i}\partial_{\bf j}V(x) = \partial_{\bf i}\partial_{{\bf j}^+(k)}V(x) =0,
\]
which proves the claim.  Thus for any $k\in\{1, \dots, N\}$ and ${\bf m}\in\{0,...,R+1\}^N$,
\[
S_t^{(k, {\bf m})} :=e^{t\mathcal{L}^{(k)}_{\bf m}}= \prod_{{\bf i}\in\otimes_{\sigma=1}^N\left((R+2)\Z + m_\sigma\right)}
e^{t\bX^2_{{\bf i}, {\bf i}^+(k)}}
\]
i.e. $S_t^{(k,{\bf m})}$ is a product semigroup.

We now claim that
\begin{equation}
\label{claim}
\mu_{r\bf G}\left(fS_t^{(k,{\bf m})}g\right) = \mu_{r\bf G}\left(gS_t^{(k,{\bf m})}f\right)
\end{equation}
for $k\in\{1, \dots, N\}$ and ${\bf m}\in\{0,...,R+1\}^N$.  Let $k=1$ and ${\bf m} = {\bf 0}$ (the other cases are similar).  Since $g$ depends on coordinates ${\bf i}$ such that $|{\bf i}|_1\leq n$, we have
\[
S^{(1,{\bf 0})}_tg(x) = \prod_{\substack{{\bf i}\in\otimes_{\sigma=1}^N\left((R+2)\Z + m_\sigma\right)\\ |{\bf i}|_1\leq n+R+2}}e^{t\bX^2_{{\bf i},
{\bf
i}^+(k)}}g(x),
\]
which is a finite product.  As a result of Lemma \ref{antisym}, we then have that
%we have that for any ${\bf i,j}\in\Z^N$
%\[
%\mu_{\bf G}\left(fX_{{\bf i},{\bf j}}^2g\right) = \mu_{\bf G}\left(gX_{{\bf i},{\bf j}}^2f\right)
%\]
%and hence
\begin{align*}
\mu_{r\bf G}\left(fS^{(1,{\bf 0})}_tg\right) &= \mu_{r\bf G}\left(f\prod_{\substack{{\bf i}\in\otimes_{\sigma=1}^N\left((R+2)\Z + m_\sigma\right)\\
|{\bf i}|_1\leq n+R+2}}e^{t\bX^2_{{\bf i}, {\bf i}^+(k)}}g\right) \\
&=\mu_{r\bf G}\left(g\prod_{\substack{{\bf i}\in\otimes_{\sigma=1}^N\left((R+2)\Z + m_\sigma\right)\\ |{\bf i}|_1\leq n+R+2}}e^{t\bX^2_{{\bf i}, {\bf
i}^+(k)}}f\right) \\
&= \mu_{r\bf G}\left(gS^{(1,{\bf 0})}_tf\right)
\end{align*}
as claimed.

To finish the proof, we will need to use the following version of the Trotter product formula (see \cite{T-Z}):

\begin{thm}
\label{zabczyk}
Let $\mathcal{H}$ and $\mathcal{E}$ be two Hilbert spaces, and $F_i\in Lip(\mathcal{E},\mathcal{E}), U_i\in Lip(\mathcal{E},
L_{HS}(\mathcal{H},\mathcal{E}))$
for $i=1,2,3$.  Let $(W(t))_{t\geq0}$ be a cylindrical Wiener process in $\mathcal{H}$.  Consider the SDEs, indexed by $i=1,2,3$, given by
\[
dY_i(t) = F_i(Y_i(t))dt + U_i(Y_i(t))dW(t), \qquad Y_i(0) = x\in \mathcal{E},
\]
and let $(\mathcal{P}^i_t)_{t\geq0}$ be the corresponding semigroups on $\mathcal{U}C_b(\mathcal{E})$.  Assume that
\[
F_3 = F_1 + F_2, \qquad U_3U_3^* = U_1U_1^* + U_2U_2^*,
\]
and that the first and second Fr\'echet derivatives of $F_i$ and $U_i$ are uniformly continuous and bounded on bounded subsets of $\mathcal{E}$.
Then
\[
\lim_{n\to\infty}\left(\mathcal{P}_{\frac{t}{n}}^1\mathcal{P}_{\frac{t}{n}}^2\right)^nf(x) = \mathcal{P}_{t}^3f(x)
\]
for all $f\in \mathbb{K}$, where $\mathbb{K}$ is the closure of $\mathcal{U}C_b^2(\mathcal{E})$ in $\mathcal{U}C_b(\mathcal{E})$, and the convergence is uniform in $x$ on any
bounded subset of $\mathcal{E}$.
\end{thm}

By above, we have that the generator of our system can be decomposed as
\[
\mathcal{L} = \frac{1}{2}\sum_{k=1}^N\sum_{{\bf m} \in \{0, \dots, R+1\}^N}\mathcal{L}_{\bf m}^{(k)}
\]
where, for  $k\in\{1, \dots, N\}$ and ${\bf m}\in\{0,...,R+1\}^N$, $\mathcal{L}_{\bf m}^{(k)}$ is the generator of the semigroup $S_t^{(k, {\bf m})}$.  The
associated SDE is given by
\[
dY(t) = A_{\bf m}^{(k)}Y(t)dt + B_{\bf m}^{(k)}(Y(t))dW(t),
\]
where $A_{\bf m}^{(k)} : E_\alpha \to E_\alpha$ and $B_{\bf m}^{(k)} \in L\left(E_\alpha, L_{HS}(E_\alpha, H)\right)$ are such that
\[
A= \sum_{k=1}^N\sum_{{\bf m}\in\{0, \dots, R+1\}^N}A_{\bf m}^{(k)}
\]
and
\[
BB^* = \sum_{k=1}^N\sum_{{\bf m}\in\{0, \dots, R+1\}^N} B_{\bf m}^{(k)}\left(B_{\bf m}^{(k)}\right)^*.
\]

We can then apply Theorem \ref{zabczyk} iteratively to get the result.  Indeed, order the set
\[
\{1, \dots, N\}\times\{0, \dots, R+1\}^N =\{ \iota_1, \dots, \iota_S\}
\]
where $S=N(R+2)^N$.  If $\iota_l = (k, {\bf m}) \in\{1, \dots, N\}\times\{0, \dots, R+1\}^N$, write
\[
A_{\bf m}^{(k)}=A_{\iota_l}, \quad B_{\bf m}^{(k)} = B_{\iota_l}, \quad \mathcal{L}_{\bf m}^{(k)} = \mathcal{L}_{\iota_l}, \quad S^{k,{\bf m}}_t =
S^{\iota_l}_t.
\]
Then define for $1\leq l \leq S$
\[
\qquad \tilde{A}_{l} := \sum_{j=1}^lA_{\iota_j}
\]
and $\tilde{B}_l\in L\left(E_\alpha, L_{HS}(E_\alpha, H)\right)$ to be such that
\[
\tilde{B}_l\tilde{B}_l^* := \sum_{j=1}^{l}B_{\iota_j}B_{\iota_j}^*.
\]
Consider the SDE
\[
d\tilde{Y}_l(t) = \tilde{A}_l\tilde{Y}_l(t)dt + \tilde{B}_l\left(\tilde{Y}_l(t)\right)dW(t),
\]
which has generator $\tilde{\mathcal{L}}_l = \sum_{j=1}^l\mathcal{L}_{\iota_j}$. Let $(\tilde{P}^l_t)_{t\geq0}$ be the semigroup on
$\mathcal{U}C_b(E_{\alpha})$ associated with
$\tilde{\mathcal{L}}_l$.
By a first application of Theorem \ref{zabczyk}, for all $f\in \mathbb{K}$, we have
\[
\lim_{n\to\infty}\left(S^{\iota_1}_{\frac{t}{n}}S^{\iota_2}_{\frac{t}{n}}\right)^nf(x) = \tilde{P}^2_tf(x)
\]
where the convergence is uniform on bounded subsets.  Moreover, by claim \eqref{claim} above and the dominated convergence theorem, we have
\begin{align}
\label{iteration1}
\mu_{r\bf G}\left(f\tilde{P}^2_tg\right) &
= \lim_{n\to\infty}\mu_{r\bf
G}\left(f\left(S^{\iota_1}_{\frac{t}{n}}S^{\iota_2}_{\frac{t}{n}}\right)^ng\right)\nonumber\\
&
= \lim_{n\to\infty}\mu_{r\bf G}\left(g\left(S^{\iota_1}_{\frac{t}{n}}S^{\iota_2}_{\frac{t}{n}}\right)^nf\right)%\nonumber\\
%&
=\mu_{r\bf G}\left(g\tilde{P}^2_tf\right)
\end{align}
for all $f,g\in\mathcal{U}C^2_b(E_{\alpha})$.  Similarly, for all $f\in \mathbb{K}$, we have
\[
\lim_{n\to\infty}\left(\tilde{P}^2_{\frac{t}{n}}S^{\iota_3}_{\frac{t}{n}}\right)^nf(x) = \tilde{P}_t^3f(x)
\]
where again the convergence is uniform on bounded sets, so that
\begin{align*}
\mu_{r\bf G}\left(f\tilde{P}_t^3g\right) &=  \lim_{n\to\infty}\mu_{r\bf G}\left(f\left(\tilde{P}^2_{\frac{t}{n}}S^{\iota_3}_{\frac{t}{n}}\right)^ng\right)\nonumber\\
& = \lim_{n\to\infty}\mu_{r\bf G}\left(g\left(\tilde{P}^2_{\frac{t}{n}}S^{\iota_3}_{\frac{t}{n}}\right)^nf\right)%\nonumber\\
%&
=\mu_{r\bf G}\left(g\tilde{P}^3_tf\right),
\end{align*}
using identities \eqref{claim} and \eqref{iteration1}.  Continuing in this manner, we see that $P_t = \tilde{P}^{S}_t$, the semigroup
corresponding to
the generator $\mathcal{L}=\sum_{j=1}^{s}\mathcal{L}_{\iota_j}$,  is such that
\[
\mu_{r\bf G}\left(fP_tg\right) = \mu_{r\bf G}\left(gP_tf\right)
\]
for all $f,g\in \mathcal{U}C^2_b(E_{\alpha})$, as required.
\end{proof}

Finally, by standard arguments, we can extend the above result to functions in $L^p(\mu_{r\bf G})$.

\begin{cor}\label{cor:SemigroupReversibility}
The semigroup $(P_t)_{t\geq0}$ acting on $\mathcal{U}C_b(E_{\alpha})$ can be extended to $L^p(\mu_{r\bf G})$ for any $p\geq1$, $r>0$.
Moreover we have
\[
\mu_{r\bf G}(fP_tg) = \mu_{r\bf G}(gP_tf)
\]
for any $f, g\in L^2(\mu_{r\bf G})$ and $r>0$.
\end{cor}

\section{Weak and strong continuity}\label{Section.5}
In this section we will show that the semigroup $(e^{-\beta t}P_t)_{t\geq 0}$, is weakly continuous for some $\beta>0$ in the sense of definition given
in \cite{[Cerrai-1994]}.
This will allow us to deduce the closedness of the generator $\mathcal{L}$ in $L^2(d\mu_{r\G})$ and the strong continuity of
$(P_t)_{t\geq 0}$.  Another approach to strong continuity of diffusion semigroups and connected questions is discussed in \cite{G-K}.

Let $\mathcal{E}$ be an arbitrary separable Hilbert space.  The following definition is found in
\cite{[Cerrai-1994]}.

\begin{defin}
\label{weak cont}
A semigroup of  bounded linear operators $(S_t)_{t\geq
0}$ defined on $\mathcal{U}C_b(\mathcal{E})$ is said to be weakly continuous
if
\begin{trivlist}
\item[(i)] the family of functions $\{S_t\phi\}_{t\geq 0}$ is
equi-uniformly continuous for every $\phi\in\mathcal{U}C_b(\mathcal{E})$;

\item[(ii)] for every $\phi\in\mathcal{U}C_b(\mathcal{E})$ and for every
compact set $K\subset H$
\begin{equation}
\label{WeakCont_1} \liml_{t\to 0}\supl_{x\in
K}|S_t\phi(x)-\phi(x)|=0;
\end{equation}
\item[(iii)] for every $\phi\in\mathcal{U}C_b(\mathcal{E})$ and for every
sequence $\{\phi_j\}_{j\in\mathbb{N}}\subset \mathcal{U}C_b(\mathcal{E})$ such
that
\begin{displaymath}
\left\{
\begin{array}{rcl}
\supl_{j}|\phi_j|_{L^{\infty}}&<&\infty,\\
\liml_{j\to\infty}\supl_{x\in K}|\phi_j(x)-\phi(x)|&=&0,\ {\mathrm\ for\ all\ compact\ }K\subset\mathcal{E},
\end{array}
\right.
\end{displaymath}
it holds that
\begin{equation}\label{WeakCont_2}
\liml_{j\to\infty}\supl_{x\in
K}|S_t\phi_j(x)-S_t\phi(x)|=0,
\end{equation}
for every compact set $K\subset E$, and furthermore the limit is uniform
in $t\geq 0$;
\item[(iv)]
there exist $\mathcal{M},\omega>0$ such that
\begin{equation}
|S_tf|_{\mathcal{U}C_b(\mathcal{E})}\leq \mathcal{M}e^{-\omega t}|f|_{\mathcal{U}C_b(\mathcal{E})},\qquad t\geq0
\label{ExponentialDecrease}
\end{equation}
for all $f\in\mathcal{U}C_b(\mathcal{E})$.
\end{trivlist}
\end{defin}

Now suppose we are in the situation of Sections \ref{The system}, \ref{solution} and \ref{Invariant measure} above.
Define $\tilde{H}\subset E_{\alpha}$ by
\[
\tilde{H} := \left\{ x\in \Omega: |x|_{\tilde{H}}^2:=\sum_{\bi\in\Z^N}x_{\bi}^2 <\infty\right\}.
\]
\begin{thm}
Assume that there exist $C_1,C_2>0$ such that
\begin{equation}
C_2 V(x)\leq|x|^2_{\tilde{H}}\leq C_1 V(x),\quad x\in \tilde{H}.
\label{ULbound_V}
\end{equation}
Then there exists $\beta>0$ such that semigroup $(\tilde P_t)_{t\geq0}:=\left(e^{-\beta
t}P_t\right)_{t\geq 0}$ is weakly continuous in $\mathcal{U}C_b(E_{\alpha})$.
\end{thm}
\begin{rem}
Assumption \eqref{ULbound_V} is satisfied if $\M$ is strictly positive definite and the coefficients of $\M$ are uniformly bounded, as in our case, though we state the result in a more general form.
\end{rem}
%{\color{red}
%Need to define what $V$ is, since we have only defined it formally in section 1, and only for $x\in E_\alpha$, not $H$.  Need to say that there is a
%natural isometry between $H$ as stated in section 1, and
%\[
%\tilde{H} := \left\{ x\in \R^{\Z^N} : \sum_{\bi}x_{\bi}^2 <\infty\right\}
%\]
%or something.  Then $V(x) = \sum_{\bi, \bj}x_{\bi}\M_{\bi, \bj}x_\bj$ is we defined for $x\in\tilde{H}$, not just formally.
%
%
%~
%
%With this definition, the first inequality is true in our situation, since we are assuming the coefficients $\M_{\bi, \bj}$ to be uniformly bounded.
%The second inequality is equivalent to the statement that $\M$ is strictly positive definite, which we are also assuming (since we must have that $\M$
%is invertible)?  So do we really need the assumption you state in this theorem?
%Done.
%
%}

\begin{proof}
First notice that there exists $q=q(\alpha)>0$, such that
\begin{equation}
P_t|Id|_{E_{\alpha}}^2(x)\leq |x|_{E_{\alpha}}^2e^{qt},\label{ExponentialGrowth}
\end{equation}
for all $x\in E_\alpha$ and $t>0$.  Indeed,
$P_t|Id|_{E_{\alpha}}^2(x)=\mathbb{E}|Y_t(x)|_{E_{\alpha}}^2$, where
$Y_t$ is a solution of equation \eqref{Ito op form}.  Inequality \eqref{ExponentialGrowth} then follows from It\^o's formula,
the boundedness of linear maps $A\in L(E_{\alpha},E_{\alpha})$ and $B\in
L(E_{\alpha},L_{HS}(H,E_{\alpha}))$ and Gronwall's lemma. %{\color{red} -do you think this needs more explanation? No}
Put $\beta=q$. We check the requirements of Definition \ref{weak cont}.

\begin{trivlist}
\item[(i)] Let $\phi\in\mathcal{U}C_b(E)$. Then for any $\varepsilon>0$
there exists $\delta(\varepsilon)>0$ such that
$|x-y|_{E_\alpha}<\delta(\varepsilon)\Rightarrow |\phi(x)-\phi(y)|<\varepsilon$.
Thus, for any $x,y\in E_{\alpha}$,
\begin{align}
\label{EquiUniContinuity}
|\tilde P_t\phi(x)-\tilde P_t\phi(y)| & \leq
e^{-qt}\mathbb{E}\left(\Ind_{\{|Y_t(x)-Y_t(y)|<\delta(\varepsilon/2)\}}|\phi(Y_t(x))-\phi(Y_t(y))|\right)\nonumber\\
&\quad + e^{-qt}\mathbb{E}\left(\Ind_{\{|Y_t(x)-Y_t(y)|\geq\delta(\varepsilon/2)\}}|\phi(Y_t(x))-\phi(Y_t(y))|\right)\nonumber\\
&\leq \frac{\varepsilon}{2}+2e^{-qt}|\phi|_{L^{\infty}}\mathbb{P}\left\{|Y_t(x)-Y_t(y)|_{E_\alpha}\geq\delta(\varepsilon/2)\right\}\nonumber\\
&\leq \frac{\varepsilon}{2}+\frac{2|\phi|_{L^{\infty}}}{\delta^2(\varepsilon/2)}|x-y|_{E_{\alpha}}^2,
\end{align}
where we have used Chebyshev's inequality.
%\begin{eqnarray}
%%|G_t\phi(x)-G_t\phi(y)|&\leq&
%%e^{-qt}\mathbb{E}|\phi(Y_t(x))-\phi(Y_t(y))|\nonumber\\
% %&\leq&
%%e^{-qt}\mathbb{E}\left(\left(\chi_{|Y_t(x)-Y_t(y)|<\delta(\varepsilon/2)} %|\phi(Y_t(x))-\phi(Y_t(y))|%\nonumber\\
%%%&+&
%%e^{-qt}\mathbb{E}
%%\chi_{|Y_t(x)-Y_t(y)|\geq\delta(\varepsilon/2)}\right)|\phi(Y_t(x))-\phi(Y_t(y))| \right)\nonumber\\
%\leq%&\leq&
%\frac{\varepsilon}{2}+2e^{-qt}|\phi|_{L^{\infty}}\mathbb{P}(|Y_t(x)-Y_t(y)|_{E_\alpha}\geq\delta(\varepsilon/2))%\nonumber\\
%&\leq&
%\frac{\varepsilon}{2}+\frac{2e^{-qt}|\phi|_{L^{\infty}}}{\delta^2(\varepsilon/2)}\mathbb{E}|Y_t(x)-Y_t(y)|_{E_{\alpha}}^2\nonumber\\
%\leq%&\leq&
%\frac{\varepsilon}{2}+\frac{2e^{-qt}|\phi|_{L^{\infty}}}{\delta^2(\varepsilon/2)}\mathbb{E}|Y_t(x-y)|_{E_{\alpha}}^2%\nonumber\\
%&=&
%\frac{\varepsilon}{2}+\frac{2e^{-qt}|\phi|_{L^{\infty}}}{\delta^2(\varepsilon/2)}e^{qt}|x-y|_{E_{\alpha}}^2\nonumber\\
%&=&
%\frac{\varepsilon}{2}+\frac{2|\phi|_{L^{\infty}}}{\delta^2(\varepsilon/2)}|x-y|_{E_{\alpha}}^2. \label{EquiUniContinuity}
%\end{eqnarray}
Choose $\delta_1(\varepsilon)$ such that
$\frac{2|\phi|_{L^{\infty}}}{\delta^2(\varepsilon/2)}\delta_1(\varepsilon)^2=\frac{\varepsilon}{2}$.
Then $|x-y|_{E_{\alpha}}<\delta_1(\varepsilon) \Rightarrow |\tilde P_t\phi(x)-\tilde P_t\phi(y)|\leq \varepsilon$, and so the first requirement holds.

\item[(ii)] Fix compact $K\subset E_{\alpha}$ and $\phi\in
\mathcal{U}C_b(E_{\alpha})$. For $\varepsilon>0$ again let $\delta(\varepsilon)>0$ be
such that $|x-y|_{E_\alpha}\leq \delta(\varepsilon)\Rightarrow
|\phi(x)-\phi(y)|\leq\varepsilon$ for any $x,y\in E_{\alpha}$.
Since
\begin{eqnarray*}
  |\tilde P_t\phi(x)-\phi(x)| &\leq& |P_t\phi(x)-\phi(x)|+ (1-e^{-qt})|P_t\phi|_{L^\infty} \\
  &\leq& |P_t\phi(x)-\phi(x)|+ (1-e^{-qt})|\phi|_{L^\infty},
\end{eqnarray*}
it is enough to show that
\begin{equation}
\label{WeakCont_1a} \liml_{t\to 0}\supl_{x\in
K}|P_t\phi(x)-\phi(x)|=0.
\end{equation}
A similar calculation to the above yields
\begin{align}
|P_t\phi(x)-\phi(x)|&\leq \mathbb{E}|\phi(Y_t(x))-\phi(x)|\nonumber\\
 &\leq \mathbb{E}\Ind_{\{|Y_t(x)-x|_{E_\alpha}\leq \delta(\varepsilon/2)\}}|\phi(Y_t(x))-\phi(x)| \nonumber\\
 &\quad +\mathbb{E}\Ind_{\{|Y_t(x)-x|_{E_\alpha}>
 \delta(\varepsilon/2)\}}|\phi(Y_t(x))-\phi(x)|\nonumber\\
 &\leq \varepsilon/2+2|\phi|_{L^{\infty}}\mathbb{P}\left\{|Y_t(x)-x|_{E_\alpha}>
 \delta(\varepsilon/2)\right\}\nonumber\\
 &\leq\varepsilon/2+\frac{2|\phi|_{L^{\infty}}}{\delta^2(\varepsilon/2)}\mathbb{E}|Y_t(x)-x|_{E_{\alpha}}^2.
\label{WC_aux_0}
\end{align}

Since $Y_t$ is a mild solution of equation
\eqref{system}, we have
$$
Y_t(x)-x=e^{At}x-x+\int_0^te^{A(t-s)}B(Y_s)dW_s.
$$
Therefore, using the It\^o isometry, for $x\in E_{\alpha}$ and $t\geq 0$, we see that
\begin{align}
    \mathbb{E}|Y_t(x)-x|_{E_{\alpha}}^2&\leq
    2|e^{At}x-x|_{E_{\alpha}}^2+2\mathbb{E}\int_0^t|e^{A(t-s)}B(Y_s)|_{L_{HS}(H,E_{\alpha})}^2ds\nonumber\\
    &\leq
    2|e^{At}-Id|_{L(E_{\alpha},E_{\alpha})}^2|x|_{E_{\alpha}}^2\nonumber\\
    & \quad +2\supl_{\tau\in [0,t]}|e^{A\tau}|_{L(E_{\alpha},E_{\alpha})}^2|B|_{L(E_{\alpha},L_{HS}(H,E_{\alpha}))}^2\intl_0^t\mathbb{E}|Y_s|_{E_{\alpha}}^2ds\nonumber\\
    &\leq 2|e^{At}-Id|_{L(E_{\alpha},E_{\alpha})}^2|x|_{E_{\alpha}}^2\nonumber\\
    &\quad + 2e^{2|A|_{L(E_{\alpha},E_{\alpha})}t}|B|_{L(E_{\alpha},L_{HS}(H,E_{\alpha}))}^2|x|_{E_{\alpha}}^2\frac{e^{qt}-1}{q},\label{WC_aux_01}
\end{align}
where the last inequality follows from \eqref{ExponentialGrowth}. Combining \eqref{WC_aux_0}
and \eqref{WC_aux_01}, we get for $t\in[0,1]$,
\[
|P_t\phi(x)-\phi(x)|\leq
\varepsilon/2+\frac{4|\phi|_{L^{\infty}}}{\delta^2(\varepsilon/2)}|x|_{E_{\alpha}}^2\Big(C(A,B,q)(e^{qt}-1)+2|e^{At}-Id|_{L(E_{\alpha},E_{\alpha})}^2\Big).
\]
%{\color{red}
%Not quite sure why $C$ doesn't depend on $t$?  How do you deal with the term $\supl_{\tau\in[0,t]}|e^{A\tau}|_{L(E_{\alpha},E_{\alpha})}^2$-is it just that
%$e^{tA}$ is a contraction semigroup?  In this case I guess the constant $C$ wont depend on $A$.
%}
%{\color{blue} Changed.}
Choose $\tau\in (0,1]$ such that
$$
\frac{2|\phi|_{L^{\infty}}}{\delta^2(\varepsilon/2)}\supl_{x\in
K}|x|_{E_{\alpha}}^2\left[C(A,B,q)(e^{q\tau}-1)+2\supl_{t\in[0,\tau]}|e^{At}-Id|_{L(E_{\alpha},E_{\alpha})}^2\right]\leq\varepsilon/2.
$$
Then for any $0\leq
t\leq\tau$,
\[
\supl_{x\in K}|P_t\phi(x)-\phi(x)|\leq\varepsilon,
\]
and \eqref{WeakCont_1a} follows.

\item[(iii)]
Fix compact $K\subset E_{\alpha}$. Define
$$
\tilde{K}=\tilde{K}(\omega)=\overline{\cupl_{t\geq 0}Y_t(K)},\quad\omega\in\Omega.
$$
We first show that $\tilde{K}$ is compact with probability $1$.  For any $\varepsilon>0$ there exist $x(1),\ldots, x(n)\in
E_{\alpha}$ such that
\begin{equation}
K\subset\cupl_{i=1}^nB_{\varepsilon/2}(x(i)).\label{WC_aux_1}
\end{equation}
Since $\tilde{H}$ is dense in $E_{\alpha}$ we can always assume that $x(i)\in
\tilde{H}$. It follows from assumption \eqref{ULbound_V} that
$V(x(i))<\infty$ for $i=1,\ldots,n$. Therefore, by It\^o's lemma and the
identity $\bX_{{\bf i},{\bf j}}V=0$, we conclude\footnote{We can assume that
exceptional set of measure $0$ in equality
\eqref{Conservation_Law} is the same for all $t\geq 0$ because we
can choose a continuous modification of the process $Y$.} that
$\mathbb{P}$-a.s.
\begin{equation}
V(Y_t(x(i)))=V(x(i)),\quad t\geq 0,\quad i=1,\ldots,n.\label{Conservation_Law}
\end{equation}
Hence, using assumption \eqref{ULbound_V} once more, we see that there
exists $C>0$ such that $\mathbb{P}$-a.s.
\begin{equation}\label{WC_aux_2}
    |Y_t(x(i))|_{\tilde{H}}\leq C\sup\limits_l |x(l)|_{\tilde{H}},\quad t\geq 0,\quad i=1,\ldots,n.
\end{equation}
Since the embedding $\tilde{H}\subset E_{\alpha}$ is compact, there exist
$y(1),\ldots, y(m)\in E_{\alpha}$ such that
\begin{equation}
\cupl_{t,i}Y_t(x(i))\subset\cupl_{l=1}^mB_{\varepsilon/2}(y(l))\label{WC_aux_3}
\end{equation}
$\mathbb{P}$-a.s.  Combining identities \eqref{WC_aux_1} and \eqref{WC_aux_3} we
deduce that
\begin{equation}
\cupl_{t\geq 0}Y_t(K)\subset
\cupl_{l=1}^mB_{\varepsilon}(y(l))\quad \label{WC_aux_4}
\end{equation}
and so $\tilde{K}$ is  compact $\mathbb{P}$-a.s. as claimed.

Now let $\phi\in\mathcal{U}C_b(\mathcal{E})$ and $\{\phi_j\}_{j\in\mathbb{N}} \subset \mathcal{U}C_b(\mathcal{E})$ be such that
$\sup_j|\phi_j|_{L^\infty} <\infty$ and
\[
\lim_{j\to\infty}\sup_{x\in K} |\phi_j(x) - \phi(x)| = 0
\]
for all compact $K\subset E_\alpha$.  Note that
\begin{eqnarray}
\supl_{x\in K}|\tilde P_t\phi_j(x)-\tilde P_t\phi(x)|\leq e^{-qt}\supl_{x\in
K}\mathbb{E}|\phi_j(Y_t(x))-\phi(Y_t(x))|\nonumber\\
\leq \mathbb{E}\supl_{y\in \tilde{K}}|\phi_j(y)-\phi(y)|,\label{WC_aux_5}
\end{eqnarray}
for all $t\geq0,j\in \mathbb{N}$.
Since $\tilde{K}$ is compact with probability $1$, we have that
$\mathbb{P}$-a.s.
\[
\supl_{y\in \tilde{K}}|\phi_j(y)-\phi(y)|\to 0\mbox{ as
}j\to\infty.\label{WC_aux_6}
\]
%Furthermore,
%\begin{equation}
%\supl_{y\in \tilde{K}}|\phi_j(y)-\phi(y)|\leq
%\supl_j|\phi_j|_{L^{\infty}}+|\phi|_{L^{\infty}}.\label{WC_aux_7}
%\end{equation}
Thus, by the dominated convergence theorem, we conclude that
%\[
%    \mathbb{E}\supl_{y\in \tilde{K}}|\phi_j(y)-\phi(y)|\to 0\mbox{ as
%}j\to\infty.
%\]
%Thus,  by \eqref{WC_aux_5}, it follows that
\[
\lim_{j\to\infty}\sup_{x\in K}|\tilde P_t\phi_j(x) - \tilde P_t\phi(x)| = 0
\]
for all compact $K\subset E_\alpha$.

\item[(iv)] Since $\tilde{P}_tf=e^{-qt}\mathbb{E}f(Y_t)$, we have that
$$
|\tilde P_tf|_{\mathcal{U}C_b(E_{\alpha})} \leq e^{-qt}|f|_{\mathcal{U}C_b(E_{\alpha})},
$$
for all $f\in \mathcal{U}C_b(E_{\alpha})$ and $t\geq0$.
\end{trivlist}
\end{proof}

\begin{cor}
%{\color{red} Do we need any assumptions?  I guess not, as long as we don't need any assumptions in the previous theorem (see my previous remarks)}.
The operator $\mathcal{L}$ is closed and the semigroup $(P_t)_{t\geq 0}$ is strongly continuous in $L^2(d\mu_{r\G})$, for all $r>0$.
\end{cor}
\begin{proof}
$\mathcal{L}$ is closed by Theorem 5.1 of \cite{[Cerrai-1994]}. Strong continuity follows from property (ii) of the definition of weak
continuity
above and a standard approximation procedure.
\end{proof}

%%%%%%%%%%%%%%%%%%%%%%%%%%%%%%%% Section.6 %%%%%%%%%%%%%%%%%%%%%%%%%%%%%%%%
\section{Symmetry in Sobolev spaces} % of order $1$}
\label{symmetry}\label{Section.6}

In this section we show that the generator $\mathcal{L}$ as given in Lemma \ref{generator} is symmetric and dissipative in some family of infinite dimensional Sobolev spaces.  In the next section this result will be useful in the proof of ergodicity of the semigroup generated by $\mathcal{L}$.
For $r>0$, we start by introducing the following Dirichlet operator:
\[
(f,L_r g)_{L^2(\mu_{r\G})}= -\sum_{{\bf k,l}\in\Z^N}\G_{\bf k, l}(\partial_{\bf k}f, \partial_{\bf l}g)_{L^2(\mu_{r\G})}
\]
where $\G = \M^{-1}$ is the covariance matrix associated to the measure $\mu_\G$, as above.
That is, on a dense domain including \(\mathcal{U}C_b^2\), we have
\begin{equation}
L_rg = \sum_{{\bf k,l}\in\Z^N}\G_{\bf k, l}\partial_{\bf k}  \partial_{\bf l}g   %+
- r^{-1}D g
\end{equation}
where
\begin{equation}
Dg \equiv \sum_{{\bf k }\in\Z^N}x_{\bf k}\partial_{\bf k} g.
\end{equation}
$D$ will play the role of the dilation generator in our set-up. We remark that
\[[D, \bX_{\bf i,j}] = 0\]
i.e. our fields are of order zero.
Thus
\[[D, \mathcal{L}]=0.\]
Note also that by a simple computation, we get
\[
\left[\sum_{{\bf k, l}\in\Z^N}\G_{\bf k, l}\partial_{\bf k}  \partial_{\bf l}, \mathcal{L}\right]=0.\]
This is because
\[ [\partial_{\bf k}, \bX_{{\bf i},{\bf j}} ]
= [\partial_{\bf k}, \sum_{{\bf l}\in\Z^N}{\bf M}_{\bf i,l}x_{\bf l}\partial_{\bf j} -   \sum_{{\bf l}\in\Z^N}{\bf M}_{\bf j,l}x_{\bf l}\partial_{\bf i}]
= {\bf M}_{\bf i,k} \partial_{\bf j} - {\bf M}_{\bf j,k} \partial_{\bf i},
\]
so that
\[ \left[\sum_{{\bf k, l}\in\Z^N}\G_{\bf k, l}\partial_{\bf k}  \partial_{\bf l}, \bX_{{\bf i},{\bf j}}\right]=0. \]

We thus obtain the following result.

\begin{prop}\label{commutativity1}
On \(\mathcal{U}C_b^4\), we have
\begin{equation}
[L_r, \bX_{{\bf i},{\bf j}}]  = 0
\end{equation}
for all ${\bf i,j} \in\Z^N$ and all $r>0$, from which it follows that
\begin{equation}
[L_r,\mathcal{L}]  = 0
\end{equation}
for all $r>0$.
\end{prop}
Keeping this in mind, we introduce the following family of  Hilbert spaces
$$
\tilde\X_r^n=\left\{f\in L^2(\mu_{r\G})\cap \mathcal{D}(L_r^n): |f|_{\tilde\X_r^n}^2 := |f|_{L^2(d\mu_{r\G})}^2+(f,(-L_r)^nf)_{L^2(d\mu_{r\G})}
<\infty\right\}%\subset L^2(d\mu_\G)
$$
equipped with the corresponding inner product
\[
(f, g)_{\tilde\X_n} = \mu_{r\G}(fg) + (f,(-L_r)^nf)_{L^2(d\mu_{r\G})},
\]
for $f, g\in\tilde\X_r^n$, where $n\in\mathbb{N}\cup \{0\}$ and $r>0$. Hence we obtain the following fact, (where besides Proposition \ref{commutativity1} we also use Lemma \ref{antisym}).

\begin{prop}\label{symmetry1}
For all $n\in\mathbb{N}\cup\{0\}$ and $r>0$, on a dense set \(\tilde{\mathcal{D}}_r^n\subset\tilde\X_r^n\), we have
\begin{equation}
(f, \mathcal{L}g)_{\tilde\X_r^n} = (\mathcal{L}f, g)_{\tilde\X_r^n}  = -\frac{1}{4}\sum_{{\bf i}\in\Z^N}\sum_{{\bf j}:{\bf i}\sim{\bf j}}(\bX_{{\bf i},{\bf j}}f, \bX_{{\bf i},{\bf j}}g)_{\tilde\X_r^n} .
\end{equation}
%with the sum over admissible (in the definition of \(\mathcal{L}\)) of pairs of indices.
\end{prop}

In the case when $n=1$, we have
\[(f, g)_{\tilde\X_r }= \mu_{r\G}(fg) + \sum\limits_{\bi, \bj\in \mathbb{Z}^N}\mu_{r\G}(\G_{\bi,\bj}\partial_\bi f \partial_\bj g)
=\mu_{r\G}(fg) + \mu_{r\G}(\G^\frac12\nabla f\cdot \G^\frac12\nabla g)
 \]
where for simplicity here and later we set $\tilde\X_r \equiv\tilde\X_r^1(=\mathbb{X}_r)$.
%Since $G$ is symmetric positive definite matrix scalar product above is correctly defined.
By induction, and using the fact that \([\G^\frac12\nabla, L_r]=\G^\frac12\nabla\), one can show that there exist non-negative constants \(a_{m,n},\ m=1,..,n\) with \(a_{n,n}>0\), such that
\[
(f, g)_{\tilde\X_r^n}= \mu_{r\G}(fg) + \sum_{m=1,..,n}a_{m,n}\mu_{r\G}((\G^\frac12\nabla)^{\otimes m} f\cdot (\G^\frac12\nabla)^{\otimes m}  g).
 \]
This motivates the introduction of the associated family $\X_r^n$ of Hilbert spaces with corresponding scalar products
\[(f, g)_{\X_r^n} \equiv \mu_{r\G}(fg) + \mu_{r\G}((\G^\frac12\nabla)^{\otimes m} f\cdot (\G^\frac12\nabla)^{\otimes m}  g).
\]
Again by induction, we conclude with the following property.
\begin{prop}\label{symmetry2}
For all $n\in\mathbb{N}\cup\{0\}$ and $r>0$, on a dense set \(\mathcal{D}_r^n\subset \mathbb{X}_r^n\), we have
\begin{equation}
(f, \mathcal{L}g)_{\X_r^n} = (\mathcal{L}f, g)_{\X_r^n}   .
\end{equation}
\end{prop}
We remark that neither of the families of spaces are orthogonal, but that the tilded one further allows a Fock-type stratification, which provides invariant subspaces other than the eigenspaces of the dilation generator \(D\).

\begin{rem}
The operator $\mathcal{L}$ is closable in $\X_r$ for all $r>0$, and its closure has a self-adjoint extension which is bounded from above.  We continue to denote this extension by the same symbol
$\mathcal{L}$. Moreover, the
self-adjoint extension $\mathcal{L}$ generates a strongly continuous semigroup $T_t=e^{t\mathcal{L}}:\X_r\to \X_r$ such that $T_t=P_t|_{\X_r}$.
\end{rem}
%\begin{proof}
%We can always assume that $K=0$ i.e. $\mathcal{L}$ is non positive. Indeed, otherwise we can consider operator $(\mathcal{L}-K Id)$ below.
%By proposition 3.3 of \cite{[MaRockner-1992]} operator $\mathcal{L}$ is closable. Its closure has non positive self-adjoint extension by the von
%Neumann-Krein
%Theorem.
%Hence, by the spectral theorem, the strongly continuous contraction semigroup $T_t=e^{t\mathcal{L}}:\X\to \X$ is well defined. Furthermore,
%$(T_t)_{t\geq 0}$ is a
%contraction on
%$L^2(\mu_\G)$ i.e. $|T_tf|_{L^2(\mu_\G)}^2\leq |f|_{L^2(\mu_\G)}^2$ for all $t\geq 0$ and $f\in \X$. Therefore it can be extended to a strongly
%continuous contraction
%semigroup $(S_t)_{t\geq 0}$, on $L^2(\mu_\G)$. Now $S_{\cdot}f$ and $P_{\cdot}f$, for $f\in \mathcal{U}C_b^3(E_{\alpha})$, are both strong solutions of
%the same
%equation. Hence
%$S_{\cdot}|_{\mathcal{U}C_b^3(E_{\alpha})}=P_{\cdot}|_{\mathcal{U}C_b^3(E_{\alpha})}$. Therefore, $S=P$. Indeed, $\mathcal{U}C_b^3(E_{\alpha})$ is dense in
%$L^2(\mu_G)$ and both $S_{\cdot}f$ and $P_{\cdot}f$ are contractions, uniquely defined by their value on the set $\mathcal{U}C_b^3(E_{\alpha})$.
%\end{proof}

%%%%%%%%%%%%%%%%%%%%%%%%%%%%%%%% Section.7 %%%%%%%%%%%%%%%%%%%%%%%%%%%%%%%%
\section{Ergodicity}\label{Section.7}
\label{ergodicity}

Before we get into general estimates, it is interesting to consider a few cases where some explicit bounds can be obtained.
First of all consider the linear functions
\[
F(x) \equiv \sum_{{\bf k}\in\Z^N} \alpha_{\bf k} x_{\bf k}.
\]
We note that
\begin{align}
\mathcal{L}F(x) %&= \sum_{{\bf i},{\bf j}}'\sum_{\bf k} \alpha_{\bf k} X_{{\bf i},{\bf j}}^2x_{\bf k} \\
%&=\sum_{{\bf i},{\bf j}}'\sum_{\bf k} \left({\bf M}_{\bf j,k}\alpha_{\bf i} - {\bf M}_{\bf i,k}\alpha_{\bf j}\right) X_{{\bf i},{\bf j}}x_{\bf k}\nonumber\\
%&=\sum_{{\bf i},{\bf j}}'\sum_{\bf k} \left(\left({\bf M}_{\bf j,k}({\bf M}_{\bf j,i}\alpha_{\bf i} - {\bf M}_{\bf i,i}\alpha_{\bf j}\right)
%-{\bf M}_{\bf i,k}\left({\bf M}_{\bf j,j}\alpha_{\bf i} - {\bf M}_{\bf i,j}\alpha_{\bf j}\right) \right)
%x_{\bf k}\nonumber\\
%&
=\frac{1}{4}\sum_{{\bf k}\in\Z^N} \left(\sum_{{\bf i},{\bf j}}'\left(\left({\bf M}_{\bf j,k}{\bf M}_{\bf j,i} - {\bf M}_{\bf i,k}{\bf M}_{\bf j,j} \right) \alpha_{\bf i}
+\left(  {\bf M}_{\bf j,k}{\bf M}_{\bf i,j} - {\bf M}_{\bf j,k}{\bf M}_{\bf i,i}  \right) \alpha_{\bf j} \right) \right)
x_{\bf k}\nonumber,
\end{align}
where the sum $\sum_{{\bf i},{\bf j}}'$ indicates that we sum over the pairs of indices as in the definition of $\mathcal{L}$.
In particular, in the case when $\M=b\mathbf{Id}, b\in(0,\infty)$, we have
\[\mathcal{L}F = -Nb^2\sum_{{\bf k}\in\Z^N} %\left(\sum_{{\bf i},{\bf j}}'\left(\left({\bf M}_{\bf j,k}{\bf M}_{\bf j,i} - {\bf M}_{\bf i,k}{\bf M}_{\bf j,j} \right) \alpha_{\bf i}
%+\left(  {\bf M}_{\bf i,k}{\bf M}_{\bf i,j} - {\bf M}_{\bf j,k}{\bf M}_{\bf i,i}  \right) \alpha_{\bf j} \right) \right)
\alpha_{\bf k}x_{\bf k}\nonumber .
\]
Since the semigroup maps the space of linear functions into itself, we conclude that
\[ \mu_{r\G}|P_tF-\mu_{r\G} F|^2\leq e^{-mt} \mu_{r\G}|F-\mu_{r\G} F|^2
\]
for $r>0$, with some $m\in(0,\infty)$ i.e. on linear functions we get exponential decay to equilibrium.
An inequality of this form on a dense set would imply a Poincar\'e inequality. One can, however, show that such an inequality cannot hold. To this end consider a sequence of functions of the following form:
\[
f_\Lambda(x) \equiv \sum_{{\bf i}\in\Lambda } x_{\bf i}^2 %+ \sum_{{\bf i}\in\Lambda^c}a_{\bf i} x_{\bf i}^2
\]
%%with
%%\[r_\Lambda^2\equiv\sum_{{\bf i}\in\Lambda } x_{\bf i}^2. \]
%and some
%\[\sum_{{\bf i}\in\Z^N}a_{\bf i}^2 <\infty\]
for a finite set $\Lambda$.  Then, for the measure with diagonal covariance matrix, we have
\[\mu_{r\G}|f_\Lambda-\mu_{r\G} f_\Lambda|^2 \geq |\Lambda|  \mu_{r\G}|x_{\bf i}^2-\mu_{r\G} x_{\bf i}^2|^2 \equiv const\cdot |\Lambda|
\]
with $|\Lambda|$ denoting cardinality of $\Lambda$.  Moreover,
\[
\mu_{r\G}\left(f_\Lambda ( -\mathcal{L} f_\Lambda)\right) =
\frac{1}{4}\sum_{\substack{{\bf i}\in \Lambda,{\bf j}\in \Lambda^c\\ {\bf i}\sim{\bf j} }} \mu_{r\G}(\bX_{{\bf i},{\bf j}}f_\Lambda)^2 = const\cdot|\partial \Lambda|.
\]
From this we see that for a suitable sequence of subsets $\Lambda$ invading the lattice,
the ratio of $\mu_{r\G}\left(f_\Lambda ( -\mathcal{L} f_\Lambda)\right)$ to $\mu_{r\G}|f_\Lambda-\mu_{r\G} f_\Lambda|^2$ converges to $0$.

In the remainder of this section we develop a strategy to obtain optimal estimates on the decay to equilibrium for more general spaces of functions, for simplicity working in the set-up when the matrix $\M$ is given by $\M=b\mathbf{Id}, \, b\in(0,\infty)$.
We show that the corresponding semigroup is ergodic with polynomial decay.

For $r>0$, first define
$$
\mathcal{A}_r(f)\equiv\left(\suml_{{\bf
i}\in\Z^N}\mu_{r\G}|\partial_\bi f|^2\right)^{1/2}.
$$

\begin{lem}\label{lem:PolynomConv}
For any $r>0$, $f\in \X_r$, $\bi\in\Z^N$ and $t>0$,
\begin{equation}\label{PolynomialConv_1}
\mu_{r\G}|\partial_\bi(P_tf)|^2\leq \frac{A^N}{t^{\frac{N}{2}}} \mathcal{A}^2_r(f),
\end{equation}
where $A=\frac{1}{b}\supl_{t>0}\sqrt{t}\intl_0^1e^{-2t(1-\cos(2\pi\beta))}d\beta$.
\end{lem}

\begin{proof}
Fix $r>0$.  It is enough to show \eqref{PolynomialConv_1} for $f\in \mathcal{U}C_b^4(E_{\alpha})$. Indeed, $\mathcal{U}C_b^4(E_{\alpha})$ is dense in $\X_r$ and
$(P_t)_{t\geq
0}$ is a contraction on $\X_r$.

Denote $f_t=P_tf$ for $t\geq 0$.
For $\bi\in\Z^N$, we can calculate that
\begin{align}
|\partial_\bi f_t|^2-P_t|\partial_\bi f|^2&=\int_0^t\frac{d}{ds}P_{t-s}|\partial_\bi f_s|^2ds\label{eqn:AuxEquation-1}\nonumber\\
&=\int_0^tP_{t-s}(-\mathcal{L}(|\partial_\bi f_s|^2)+2\partial_\bi f_s\mathcal{L}\partial_\bi f_s+2\partial_\bi
f_s[\partial_\bi,\mathcal{L}]f_s)ds\nonumber\\
&=\int_0^tP_{t-s}\Big(-\suml_{\substack{{\bf m,l}\in\Z^N\\{\bf m}\sim{\bf l}}}|\bX_{\bf m,l}(\partial_\bi f_s)|^2\nonumber\\
& \quad +2b\partial_\bi f_s\suml_{k=1}^N(-b\partial_\bi f_s+\bX_{\bi,\bi^-(k)}\partial_{\bi^-(k)}f_s+\bX_{\bi,\bi^+(k)}\partial_{\bi^+(k)}f_s)\Big)ds .
\end{align}
%where we have used \eqref{generator commutator}.
Integrating \eqref{eqn:AuxEquation-1} with respect to the invariant measure
$\mu_{r\G}$ yields
\begin{align}
\mu_{r\G}|\partial_\bi f_t|^2&-\mu_{r\G}|\partial_\bi f|^2 = \int_0^t\Big(-\suml_{\substack{{\bf m,l}\in\Z^N\\{\bf m}\sim{\bf l}}}
\mu_{r\G}|\bX_{\bf m,l}(\partial_\bi f_s)|^2\nonumber\\
&- 2Nb^2\mu_{r\G}|\partial_\bi f_s|^2+2b\suml_{k=1}^N\mu_{r\G}(\partial_\bi f_s\bX_{\bi,\bi^-(k)}\partial_{\bi^-(k)}f_s)\nonumber\\
&+ 2b\suml_{k=1}^N\mu_{r\G}(\partial_\bi f_s\bX_{\bi,\bi^+(k)}\partial_{\bi^+(k)}f_s)\Big)ds.
\end{align}
By Lemma \ref{antisym}, the operators $\bX_{{\bf i},{\bf j}},\bi,\bj\in\Z^N$, are anti-symmetric in
$L^2(\mu_{r\G})$. Therefore
\begin{eqnarray}
\mu_{r\G}|\partial_\bi f_t|^2&-&\mu_{r\G}|\partial_\bi f|^2=\int_0^t\Big(-\suml_{\substack{{\bf m,l}\in\Z^N\\{\bf m}\sim{\bf l}}}
\mu_{r\G}|\bX_{\bf m,l}(\partial_\bi f_s)|^2\nonumber\\
&-&2Nb^2\mu_{r\G}|\partial_\bi f_s|^2-2b\suml_{k=1}^N\mu_{r\G}(\partial_{\bi^-(k)}f_s\bX_{\bi,\bi^-(k)}\partial_\bi f_s)\nonumber\\
&-&2b\suml_{k=1}^N\mu_{r\G}(\partial_{\bi^+(k)}f_s\bX_{\bi,\bi^+(k)}\partial_\bi f_s)\Big)ds.
\end{eqnarray}
Hence, by Young's inequality we deduce that
\begin{align}
&\mu_{r\G}|\partial_\bi f_t|^2-\mu_{r\G}|\partial_\bi f|^2\nonumber\\
&\quad \leq\int_0^t\Big(-\sum_{\substack{{\bf m,l}\in\Z^N\\{\bf m}\sim{\bf l}}}
\mu_{r\G}|\bX_{\bf m,l}(\partial_\bi f_s)|^2 - 2Nb^2\mu_{r\G}|\partial_\bi f_s|^2 + \suml_{k=1}^Nb^2\mu_{r\G}|\partial_{\bi^-(k)}f_s|^2 \nonumber\\
&\qquad + \suml_{k=1}^N\mu_{r\G}|\bX_{\bi,\bi^-(k)}\partial_\bi f_s|^2 + b^2\mu_{r\G}|\partial_{\bi ^+(k)}f_s|^2+\mu_{r\G}|\bX_{\bi, \bi^+(k)}\partial_\bi f_s|^2\Big)ds\nonumber\\
&\quad \leq\int_0^tb^2\suml_{k=1}^N\Big(\mu_{r\G}|\partial_{\bi^-(k)}f_s|^2+\mu_{r\G}|\partial_{\bi^+(k)}f_s|^2-2\mu_{r\G}|\partial_\bi
f_s|^2\Big)ds.\label{ComparBound_1}
\end{align}
Let  $\triangle$  denote the Laplacian on the lattice $\Z^N$ and set $F(\bi,t)=\mu_{r\G}|\partial_\bi(P_tf)|^2$ for $t\geq 0,\bi\in\Z^N$.
Then we can rewrite \eqref{ComparBound_1} as
\begin{equation}
F(t)\leq
F(0)+\intl_0^tb^2\triangle F(s)\,ds,\qquad t\in[0,\infty).\label{ComparBound_2}
\end{equation}
Hence, by the positivity of the semigroup $(e^{tb^2\triangle})_{t\geq 0}$, and Duhamel's principle, we can conclude that
\begin{equation}
F(t)\leq e^{tb^2\triangle} F(0)
\end{equation}
for $t\in[0, \infty)$.  By taking the Fourier transform, we can see that this is equivalent to
\begin{equation}
\mu_{r\G}|\partial_{\bf i}(P_tf)|^2\leq\suml_{{\bf l}\in\Z^N}\mu_{r\G}(|\partial_{\bf l}f|^2)c_{{\bf i}+{\bf l}}(\phi^{b^2t}),\label{ComparBound_3}
\end{equation}
where
$$
c_{\bf k}(\phi^{t})=\intl_{[0,1]^N}\varphi^{t}({\bf\alpha})\cos\left(2\pi \suml_{l=1}^Nk_l\alpha_l\right)d\alpha_1\ldots d\alpha_N,
$$
is the Fourier coefficient of the function $\varphi^{t}({\bf\alpha})=\exp(-2t\suml_{n=1}^N(1-\cos(2\pi\alpha_n)))$, ${\bf\alpha}=(\alpha_1,\ldots,\alpha_N)\in\mathbb{R}^N$.  The coefficient $c_{\bf k}(\varphi^t), {\bf k}\in\Z^N$ can then be bounded above by
\begin{equation}
|c_{\bf k}(\phi^t)|\leq \intl_{[0,1]^N}\varphi^t({\bf\alpha})d{\bf\alpha}=\left(\int_0^1e^{-2t(1-\cos(2\pi\beta))}d\beta\right)^N\label{ComparBound_4},
\end{equation}
and the result follows.
\end{proof}
Now define

$$
\mathcal{B}_r(f)\equiv\left(\suml_{{\bf i}\in\Z^N}\big(\mu_{r\G}|\partial_\bi
f|^2\big)^{\frac{1}{2}}\right)
$$
for $r>0$.  The we have the following convergence result.

\begin{cor}\label{cor:ConvergenceRate}
For $r>0$ and $f\in \X_r$ with $\mathcal{B}_r(f)<\infty$, we have
\begin{equation}\label{ConvergenceRate_1}
\suml_{{\bf i}\in\Z^N}\mu_{r\G}|\partial_\bi(P_tf)|^2\leq \frac{A^\frac{N}{2} }{t^{\frac{N}{4}}} \, \mathcal{A}_r(f) \mathcal{B}_r(f).
\end{equation}
Furthermore, there exists a constant $C\in(0, \infty)$ such that
\begin{equation}\label{ConvergenceRate_2ls}
\mu_{r\G}\left((P_tf)^2\log\frac{(P_tf)^2}{\mu_{r\G}(P_tf)^2 }\right)\leq C \frac{A^\frac{N}{2} }{t^{\frac{N}{4}}}\,\mathcal{A}_r(f) \mathcal{B}_r(f),
\end{equation}
and hence
\begin{equation}\label{ConvergenceRate_2}
\mu_{r\G}(P_tf-\mu_{r\G}(f))^2\leq C \frac{A^\frac{N}{2} }{t^{\frac{N}{4}}}\mathcal{A}_r(f) \mathcal{B}_r(f),
\end{equation}
i.e. our system is ergodic with polynomial rate of convergence.
\end{cor}
\begin{proof}
By Proposition \ref{symmetry1}, we have that the semigroup $P_t$ on $\X_r$ is symmetric.  Thus we have
\begin{eqnarray}
  \suml_{{\bf i}\in\Z^N}\mu_{r\G}|\partial_\bi (P_tf)|^2 &=& \suml_{{\bf i}\in\Z^N}\mu_{r\G}(\partial_\bi f\partial_\bi P_{2t}f) \nonumber\\
  &\leq& \suml_{{\bf i}\in\Z^N}\left(\mu_{r\G}|\partial_\bi f|^2\right)^{\frac{1}{2}}\left(\mu_{r\G}|\partial_\bi P_{2t}f|^2\right)^{\frac{1}{2}}
  \nonumber\\
  &\leq& \left(\suml_{{\bf i}\in\Z^N}\big(\mu_{r\G}|\partial_\bi f|^2\big)^{\frac{1}{2}}\right)\supl_{{\bf j}\in\Z^N}\left(\mu_{r\G}|\partial_\bj
  P_{2t}f|^2\right)^{1/2}. \label{AuxConvergenceRate_1}
\end{eqnarray}
Combining \eqref{AuxConvergenceRate_1} with \eqref{PolynomialConv_1} we immediately arrive at the estimate \eqref{ConvergenceRate_1}.
Now inequalities \eqref{ConvergenceRate_2ls} and \eqref{ConvergenceRate_2} follow from the logarithmic Sobolev and Poincar\'{e} inequalities for the Gaussian measure $\mu_{r\G}$.
\end{proof}

\begin{rem}
The convergence in Lemma \ref{lem:PolynomConv} cannot be improved, while the rate of convergence in Corollary \ref{cor:ConvergenceRate} is
not far from optimal. Indeed, let $W({\bf k},t)=P_t(x_{\bf k}^2)$ for $t\geq 0$ and ${\bf k}\in\Z^N$.
Then $\mathcal{L}x_{\bf k}^2=b^2\suml_{m=1}^N(x_{{\bf k}^+(m)}^2+x_{{\bf k}^-(m)}^2-2x_{\bf k}^2)$, so that,
$$
\frac{\partial W}{\partial t}=b^2\triangle W,
$$
where as above $\triangle$ denotes the discrete Laplacian on $\Z^N$.
Thus
\begin{equation}
W(t)=e^{tb^2\triangle}W(0),\qquad t\geq 0,\label{eqn:PolynomialGrowth}
\end{equation}
so that convergence in the
Lemma \ref{lem:PolynomConv} is precise (see the end of the proof of the Lemma \ref{lem:PolynomConv}).
Furthermore,  using \eqref{eqn:PolynomialGrowth} it is possible to explicitly calculate $\mu_{r\G}(P_tx_{\bf k}^2-\mu_{r\G}(x_{\bf k}^2))^2,t\geq 0$ and show that this expression converges to $0$ polynomially. Hence the operator $\mathcal{L}$ acting on $\mathbb{X}_r$ does not have a spectral gap.
\end{rem}

The following result shows that the class of functions for which the system is ergodic is larger than the
one considered in Corollary \ref{cor:ConvergenceRate}.

\begin{prop}\label{prop:Ergodicity}
The semigroup $(P_t)_{t\geq0}$ %on $L^2(\mu_\G)$
is ergodic in the Orlicz space $L_\Psi(\mu_{r\G})$ for all $r>0$, with $\Psi(s)\equiv s^2\log(1+s^2)$, in the sense that
\[
\|P_tf-\mu_{r\G} f\|_{L_\Psi(\mu_{r\G})}\to 0
\]
as $t\to \infty$, for any $f\in L_\Psi(\mu_{r\G})$ and $r>0$.  Furthermore, for all $f\in \X_r$, $|P_tf-\mu_{r\G}f|_{\X_r}\to 0$ as $t\to \infty$.
\end{prop}
\begin{proof}
For $f\in  \X_r\cap\left\{f\in L_\Psi(\mu_{r\G}): \suml_{{\bf i}\in\Z^N}\Big(\mu_{r\G}|\partial_\bi f|^2\Big)^{\frac{1}{2}}<\infty\right\}$ the
result follows from Corollary
\ref{cor:ConvergenceRate}. Now, it is enough to notice that such a set of functions is dense in $L_\Psi(\mu_{r\G})$ (resp. in $\X_r$) for the natural topology  and
$P_t$ is a
contraction on $L_\Psi(\mu_{r\G})$ (resp. in $\X_r$).
\end{proof}
%\begin{rem}
%Measures $\mu_{r\G}$, $r>0$ are invariant w.r.t. semigroup $P_t$, $t\geq 0$.
%Consequently, the results of the section \ref{Section.7} are also valid for the measures $\mu_{r\G}$, $r>0$.
%The proof is the same as in the case of $\mu_{\G}$.
%\end{rem}

%%%%%%%%%%%%%%%%%%%%%%%%%%%%%%%% Section.8 %%%%%%%%%%%%%%%%%%%%%%%%%%%%%%%%
\section{Liggett-Nash-type inequalities} \label{Section.8}
In this section we will show how to deduce Liggett-Nash-type inequalities from the results of the previous section.
\begin{thm}
There exist constants $C_1, C_2\in(0,\infty)$ such that for all $r>0$ and $f\in \X_r\cap \mathcal{D}(\mathcal{L})$ with $\mathcal{B}_r(f)<\infty$,
\begin{equation}\label{LiggettInequality_I}
\mu_{r\G}(f-\mu_{r\G}(f))^2\leq C_1\left(-\mathcal{L}f,f\right)_{L^2(\mu_{r\G})}^{\frac{N}{N+4}}
\left(\mathcal{A}_r(f) \mathcal{B}_r(f)
\right)^{\frac{4}{N+4}}
\end{equation}
and
\begin{equation}\label{NashInequality}
\left[\mathcal{A}_r(f)
\right]^{2+\frac{4}{N}}
\leq C_2\mathcal{B}_r(f)^{\frac{4}{N}}\suml_{{\bf i}\in\Z^N}\mu_{r\G}(\partial_\bi f\partial_\bi (-\mathcal{L}f)).
\end{equation}
\end{thm}

\begin{rem}
Note that inequality \eqref{NashInequality} can be considered as an  analogue of the Nash inequality in $\R^N$ (see \cite{[Nash-58]}, p.936). Indeed, such an
inequality takes the form
$$
|u|_{L^2(\mathbb{R}^N)}^{2+\frac{4}{N}}\leq C(-\Delta u,u)_{L^2(\mathbb{R}^n)}|u|_{L^1(\mathbb{R}^N)}^{\frac{4}{N}}, \quad u\in L^1(\mathbb{R}^n)\cap
W^{1,2}(\mathbb{R}^n),
$$
for some constant $C>0$, and where $\Delta$ is the standard Laplacian on $\R^N$.
The main difference is that the natural space for our operator $\mathcal{L}$ is $\X_r$ instead of $L^2$.
\end{rem}
\begin{proof}
Inequality \eqref{LiggettInequality_I} immediately follows from \eqref{ConvergenceRate_2}, Corollary \ref{cor:SemigroupReversibility} and part (b) of
Theorem 2.2 of \cite{[Liggett-1991]}.

To see \eqref{NashInequality} let $f_t=P_tf$ as usual.
We then have, by H\"{o}lder's inequality and Lemma \ref{lem:PolynomConv}, that
\begin{eqnarray}
  \suml_{\bi\in\Z^N}\mu_{r\G}(\partial_\bi f\partial_\bi f_t) &\leq& \left(\suml_{\bi\in\Z^N}\mu_{r\G}|\partial_\bi
  f|^2\right)^{\frac{1}{2}}
  \left(\suml_{\bi\in\Z^N}\mu_{r\G}|\partial_\bi f_t|^2\right)^{\frac{1}{2}}\nonumber\\
&\leq&\frac{A^{\frac{N}{4}}\left(\suml_{{\bf
i}\in\Z^N}\Big(\mu_{r\G}|\partial_\bi f|^2\Big)^{\frac{1}{2}}\right)^{\frac{1}{2}}\left(\suml_{\bi \in\Z^N}\mu_{r\G}|\partial_\bi
f|^2\right)^{\frac{3}{4}}}{t^{\frac{N}{8}}}.
\label{LigAux_1}
\end{eqnarray}

Furthermore, note that
\begin{equation}\label{LigAux_2}
 \suml_{\bi \in\Z^N}\mu_{r\G}(\partial_\bi f\partial_\bi f_t)=\suml_{\bi \in\Z^N}\mu_{r\G}|\partial_\bi f|^2+
 \int_0^t\suml_{\bi \in\Z^N}\mu_{r\G}(\partial_\bi f\partial_\bi(\mathcal{L}f_s)) ds.
\end{equation}
Define $\phi(s)=\suml_{\bi \in\Z^N}\mu_{r\G}(\partial_\bi f\partial_\bi (\mathcal{L}f_s))$ for $s\geq 0$. $\mathcal{L}$ is symmetric in $\X_r$, so
\[\phi(s)=\suml_{\bi\in\Z^N}\mu_{r\G}(\partial_\bi(\mathcal{L}f)\partial_\bi f_s), \qquad s\geq0.
\]
We can then calculate that
\begin{eqnarray}
  \phi'(s) &=&
  \suml_{\bi\in\Z^N}\mu_{r\G}(\partial_\bi(\mathcal{L}f)\partial_\bi(\mathcal{L}f_s))=
 \suml_{\bi\in\Z^N}\mu_{r\G}(\partial_\bi(\mathcal{L}f)\partial_\bi(P_s\mathcal{L}f))
  \nonumber\\
  &=& \suml_{\bi\in\Z^N}\mu_{r\G}(\partial_\bi(\mathcal{L}P_{\frac{s}{2}}f)\partial_\bi(\mathcal{L}P_{\frac{s}{2}}f))\geq 0\nonumber
\end{eqnarray}
for all $s\geq0$.  Consequently,
\begin{equation}\label{LigAux_3}
    \phi(t)=\suml_{\bi\in\Z^N}\mu_{r\G}(\partial_\bi(\mathcal{L}f)\partial_\bi f_t)\geq
    \phi(0)=\suml_{\bi\in\Z^N}\mu_{r\G}(\partial_\bi(\mathcal{L}f)\partial_\bi f)
\end{equation}
for all $t\geq0$.  Using \eqref{LigAux_3} in \eqref{LigAux_2} yields
\[\label{LigAux_4}
\suml_{\bi\in\Z^N}\mu_{r\G}(\partial_\bi f\partial_\bi f_t)
\geq\suml_{\bi\in\Z^N}\mu_{r\G}|\partial_\bi f|^2-t\suml_{\bi\in\Z^N}\mu_{r\G}(\partial_\bi(-\mathcal{L}f)\partial_\bi f)
\]
for $t\geq0$.  Therefore, using this in  \eqref{LigAux_1}, we obtain
\begin{eqnarray}\label{LigAux_5}
    \suml_{\bi \in\Z^N}\mu_{r\G}|\partial_\bi f|^2&\leq& t\suml_{\bi \in\Z^N}\mu_{r\G}(\partial_\bi(-\mathcal{L}f)\partial_\bi
    f)\nonumber\\
    &+&\frac{A^{\frac{N}{4}}\left(\suml_{{\bf
    i}\in\Z^N}\Big(\mu_{r\G}|\partial_\bi f|^2\Big)^{\frac{1}{2}}\right)^{\frac{1}{2}}\left(\suml_{\bi\in\Z^N}\mu_{r\G}|\partial_\bi
    f|^2\right)^{\frac{3}{4}}}{t^{\frac{N}{8}}}.
\end{eqnarray}
Optimization of the right-hand side of \eqref{LigAux_5} with respect to $t$ leads to \eqref{NashInequality}.
\end{proof}

%%%%%%%%%%%%%%%%%%%%%%%%%%%%%%%% Section.9 %%%%%%%%%%%%%%%%%%%%%%%%%%%%%%%%
\section{Phase transition in stochastic dynamics} \label{Section.9}

In this section we consider a family of stochastic dynamics defined by the following generators
\[
\mathcal{L}\equiv \sum_{{\bf k}\in\mathbb{Z}^N} \bX_{\Xi+{\bf k}}^2 - \beta D,
\]
where $\beta\in[0,\infty)$, and for a finite subset $\Xi\subset\mathbb{Z}^N$ we set
\[ \bX_{\Xi+\bf{k}} \equiv  \sum_{\substack{\bi,\bj \in \Xi+\bf{k} \\ \bj\sim\bi}} a_{\bi\bj}\bX_{\bi, \bj}
\]
with some constants $a_{\bi\bj}=a_{\bi+\bf{k},\bj+\bf{k}}\in\mathbb{R}$.
As a special case, we can take $\Xi$ to be a set of neighbouring points in the lattice and $\beta=0$, which includes the model studied earlier.
Define as usual $P_t\equiv e^{t\mathcal{L}}$. Then, with $f_s\equiv P_sf$, we have the following simple computation:

\begin{align}
\frac{d}{ds}P_{t-s}|\nabla f_s|^2 &= P_{t-s}\left( -\mathcal{L}|\nabla f_s|^2 + 2\nabla f_s \nabla \mathcal{L}f_s \right)\\
&= P_{t-s}\left( -2\sum_{\bf{l},\bf{k}}|\bX_{\Xi+\bf{k}}\nabla_{\bf{l}} f_s|^2 + 2\sum_{\bf{l},\bf{k}} \nabla_{\bf{l}} f_s [\nabla_{\bf{l}}, \bX_{\Xi+{\bf k}}^2]  f_s  - 2\beta |\nabla f_s|^2\right) \nonumber \\
&= P_{t-s}\left( -2\sum_{\bf{l},\bf{k}}|\bX_{\Xi+\bf{k}}\nabla_{\bf{l}} f_s|^2   \right. \nonumber \\
&\phantom{AAAAA}\left.  + 2\sum_{\bf{l}}\sum_{\bf{k}}\sum_{\substack{\bi,\bj \in \Xi+\bf{k}\\ \bj\sim\bi}}   a_{\bi\bj}
\nabla_{\bf{l}}f_s \left\{\bX_{\Xi+{\bf k}}, [\nabla_{\bf{l}}, \bX_{\bi,\bj} ]\right\}  f_s  - 2\beta |\nabla f_s|^2\right),
\nonumber
\end{align}
where $\{\cdot, \cdot\}$ denotes the anti-commutator.  Now, since
\begin{align}
% 2\sum_{\bf{l}}\sum_{\bf{k}}\sum_{\bi,\bj \in A+\bf{k}, \bj\sim\bi}   a_{\bi\bj}
%\nabla_{\bf{l}}f_s
\{\bX_{\Xi+{\bf k}}, [\nabla_{\bf{l}}, \bX_{\bi,\bj} ]\}  f_s = 2 \bX_{\Xi+{\bf k}} (\delta_{\bf{l}\bj}\nabla_{\bi}- \delta_{\bf{l}\bi}\nabla_{\bj} )
+  \sum_{\substack{\bi',\bj' \in \Xi+\bf{k}\\ \bj'\sim\bi'}} a_{\bi'\bj'}(\delta_{\bf{l}\bj}\delta_{\bi\bj'}\nabla_{\bi'}- \delta_{\bf{l}\bi}\delta_{\bi\bi'}\nabla_{\bj'})\nonumber
\end{align}
there are constants $\varepsilon\in(0,2)$ and $\eta\in\mathbb{R}$ such that
\begin{align}
\frac{d}{ds}P_{t-s}|\nabla f_s|^2 &\leq
P_{t-s}\left( -(2-\varepsilon)\sum_{\bf{l},\bf{k}}|\bX_{\Xi+\bf{k}}\nabla_{\bf{l}} f_s|^2  - 2(\beta-\eta) |\nabla f_s|^2\right)\nonumber \\
 &\leq - 2(\beta-\eta)  P_{t-s}|\nabla f_s|^2.
\end{align}
Integrating this differential inequality, we obtain
\[
|\nabla f_t|^2 \leq e^{- 2(\beta-\eta) t} P_t|\nabla f|^2 .
\]
In the case when $\Xi$ is a two point set, combining this with our analysis in previous section we conclude with the following result.

\begin{thm}
A stochastic system described by the family of generators
$$\mathcal{L}_\beta \equiv \sum_{\bi\sim\bj}\bX_{\bi,\bj}^2 -\beta D,$$
with $\beta\in[0,\infty)$, undergoes a phase transition at some $\beta_c\in[0,\infty)$.  That is, for
$\beta>\beta_c$ it decays to equilibrium exponentially fast, while for $\beta\in[0, \beta_c)$
the decay to equilibrium (for certain cylinder functions) can only be algebraic.
\end{thm}

%%%%%%%%%%%%%%%%%%%%%%%%%%%%%%%% Section.10 %%%%%%%%%%%%%%%%%%%%%%%%%%%%%%%%
\section{Homogenisation} \label{Section.10}

In Section \ref{ergodicity} it was shown that the semigroup $(P_t)_{t\geq 0}$ with generator $\mathcal{L}$ given by \eqref{eqn:GeneratorMainExample} is ergodic. We can therefore apply Theorem 1.8 of \cite{[KV-1986]} (see also \cite{K-O}) to conclude that the following functional CLT holds.
\begin{prop}
Let $\mathcal{L}$ be given by \eqref{eqn:GeneratorMainExample}, and $(Y_t)_{t\geq 0}$ be the corresponding Markov process.  Suppose $F\in \mathcal{D}((-\mathcal{L})^{-\frac{1}{2}})$ is such that $\mu_{r\G}(F)=0$, where $\mu_{r\G}$ is as above, with $r>0$.  Let $\mathbf{P}^{\mu_{r\G}}$ be the probability measure corresponding to the stationary Markov process with the same transition functions as $Y_t$, and $(\mathcal{G}_t)_{t\geq0} = (\sigma\{Y_s,s\leq t\})_{t\geq 0}$ be the filtration generated by $Y_t$.   Then there exists a square integrable martingale
$(M_t)_{t\geq 0}$ on the probability space $(\Omega,(\mathcal{G}_t)_{t\geq 0},\mathbf{P}^{\mu_{r\G}})$ with stationary increments such that $M_0=0$ and
\[
\liml_{t\to\infty}\frac{1}{\sqrt{t}}\supl_{0\leq s\leq t}\left|\int_0^t F(Y_s)\,ds-M_s\right|=0,
\]
in probability with respect to $\mathbf{P}^{\mu_{r\G}}$. Moreover,
\[
\liml_{t\to\infty}\frac{1}{t}\mathbb{E}^{\mathbf{P}^{\mu_{r\G}}}|Y_t-M_t|^2=0.
\]
\end{prop}

\paragraph{Acknowledgements:}
We would like to thank Z. Brze\'{z}niak and B. Go{\l}dys for useful remarks and attention to the work.

\bibliography{Erg280410}

\end{document}